\documentclass[11pt, leqno]{article}
\usepackage{amsmath,amssymb,latexsym,amsthm}
\usepackage[T1]{fontenc}
\usepackage{enumitem}
\usepackage{hyperref} 

\DeclareSymbolFont{symbolsC}{U}{txsyc}{m}{n}
\DeclareMathSymbol{\notniFromTxfonts}{\mathrel}{symbolsC}{61}

\newtheorem{thm}{Theorem}
\newtheorem{lem}[thm]{Lemma}

\newtheorem{dfn}[thm]{Definition}
\newtheorem{exm}{Example}
\newcommand{\bea}{\begin{eqnarray}}
\newcommand{\eea}{\end{eqnarray}}
\newcommand{\bean}{\begin{eqnarray*}}
\newcommand{\eean}{\end{eqnarray*}}
\newcommand{\beq}{\begin{equation}}
\newcommand{\eeq}{\end{equation}}
\newcommand{\bac}{\begin{array}{c}}
\newcommand{\ball}{\begin{array}{ll}}
\newcommand{\ea}{\end{array}}

\newcommand{\bbR}{{\mathbb R}}

\newcommand{\OO}{{\mathcal O}}
\def\({\left(}
\def\){\right)}

\def\bl{\left\{}
\def\br{\right\}}
\def\ml{\left|}
\def\mr{\right|}

\def\wt{\widetilde}

\def\ep{\varepsilon}

\def\dia{{\rm diam}}

\def\ul{\underline}

\begin{document}

\title{On the Whitney distortion extension problem for $C^m(\mathbb R^n)$ and $C^{\infty}(\mathbb R^n)$ and its applications to interpolation and alignment of data in $\mathbb R^n$}
\author{Steven B. Damelin \thanks{Mathematical Reviews, 416 Fourth Street, Ann Arbor, MI 48103, USA\, email: damelin@umich.edu} \and Charles Fefferman, \thanks{Department of Mathematics; Fine Hall, Washington Road, Princeton NJ 08544-1000 USA,\, email: cf@math.princeton.edu}}
\date{}
\maketitle

\thispagestyle{empty}
\parskip=10pt

\begin{abstract}
Let $n,m\geq 1$, $U\subset\mathbb R^n$ open. In this paper we provide a sharp solution to the following Whitney distortion extension problems: 
(a)  Let $\phi:U\to \mathbb R^n$ be a $C^m$ map. If $E\subset U$ is compact (with some geometry) and the restriction of $\phi$ to $E$ is an almost isometry 
with small distortion,  how to decide when there exists a $C^m(\mathbb R^n)$ one-to-one and onto almost isometry $\Phi:\mathbb R^n\to \mathbb R^n$ with small distortion
which agrees with $\phi$ in a neighborhood of $E$ and a Euclidean motion $A:\mathbb R^n\to \mathbb R^n$ away from $E$.  (b)  Let $\phi:U\to \mathbb R^n$ be $C^{\infty}$ map. 
If $E\subset U$ is compact (with some geometry) and the restriction of $\phi$ to $E$ is an almost isometry with small distortion, how to decide when there exists a $C^{\infty}(\mathbb R^n)$ 
one-to-one and onto almost isometry $\Phi:\mathbb R^n\to \mathbb R^n$ with small distortion
which agrees with $\phi$ in a neighborhood of $E$ and a Euclidean motion $A:\mathbb R^n\to \mathbb R^n$ away from $E$.
Our results complement those of \cite{FD1,FD3} where there, $E$ is a finite set. In this case, the problem above is also a problem of interpolation and alignment of data in $\mathbb R^n$. 
The work in this paper appears in the research memoir \cite{SDam}.
\end{abstract}
\vskip 0.1in

\medskip

\noindent AMS-MSC Classification: 58C25, 42B35, 94A08, 94C30, 41A05, 68Q25, 30E05, 26E10, 68Q17, 
\vskip 0.2in
\tableofcontents

\section{Whitney's extension problem.}
\setcounter{thm}{0}
\setcounter{equation}{0}
We will work here and throughout in Euclidean space $\mathbb R^n$, $n\geq 1$  $|.|$ will denote the Euclidean norm on $\mathbb R^n$.
\medskip

{\it Question 1}\, Let $m\geq 1$ and let $\phi:E\to \mathbb R$ be a function defined on an arbitrary set $E\subset \mathbb R^n$. How can one decide whether $\phi$ extends to a function
 $\Phi:\mathbb R^n\to \mathbb R$ which agrees with $\phi$ on $E$ and is in $C^m(\mathbb R^n)$, the space of functions from $\mathbb R^n$ to $\mathbb R$ whose derivatives of order 
$m$ are continuous and bounded.

For $n=1$ and $E$ compact, this is the well-known Whitney extension problem. See \cite{W,W1,W2}. Continued progress on this problem was made by G. Glaeser \cite{Gl}, Y. Brudnyi and P. Shvartsman \cite{B,BS1,BS2,BS3,BS4} and E. Bierstone, P. Milman and W. Pawlucki \cite{BMP}. See 
also N. Zobin \cite{Z,Z1}, B. Klartag and N. Zobin \cite{KZ} and 
E. LeGryer \cite{Le} for related work. C. Fefferman and his collaborators, A. Israel, B. Klartag, G. Lui, S. Mitter and H. Narayanan \cite{F1,F3,F4,F5,F6,KF,KF1,KF2,KF3, FIL1,FIL2,FIL3,FIL4,FIL5,FMH,I} have given a complete solution to this problem and have extended and 
generalized it in many ways. 

\subsection{Whitney's distortion extension problem for $C^m(\mathbb R^n)$ and $C^{\infty}(\mathbb R^n)$.}

In this paper, following from the work of \cite{DF4,FD1,FD3},  we are interested in studying the following Whitney extension problems which henceworth we will call {\it Whitney distortion extension problems for $C^m(\mathbb R^n)$ and $C^{\infty}(\mathbb R^n)$.} These are given in: 
\medskip

{\it Question 2}\, (a)  Let $U\subset\mathbb R^n$ be open and $\phi:U\to \mathbb R^n$ be a $C^m$ map. If $E\subset U$ is compact (with some geometry) and the restriction of $\phi$ to $E$ is an almost isometry with small distortion,  how to decide when there exists a $C^m(\mathbb R^n)$ one-to-one and onto almost isometry $\Phi:\mathbb R^n\to \mathbb R^n$ with small distortion
which agrees with $\phi$ in a neighborhood of $E$.  (b)  Let $\phi:U\to \mathbb R^n$ be $C^{\infty}$ map. If $E\subset U$ is compact (with some geometry) and the restriction of $\phi$ to $E$ is an almost isometry with small distortion, how to decide when there exists a $C^{\infty}(\mathbb R^n)$ one-to-one and onto almost isometry $\Phi:\mathbb R^n\to \mathbb R^n$ with small distortion
which agrees with $\phi$ in a neighborhood of $E$. Our results complement those of \cite{FD1,FD3} where there, $E$ is a finite set. In this case, as we will explain in detail below, the problem above is also a problem of interpolation and alignment of data in $\mathbb R^n$. In this paper, we will provide a sharp solution to Question 2 in the form of our main result which is given in our next  section.

\section{Main result of this paper.}


\setcounter{thm}{0}
\setcounter{equation}{0}

In this section, we state our main result which provides a sharp answer to Question 2. We need some quick notation and assumptions.

\subsection{Notations and assumptions.}

Unless stated otherwise, we will henceforth work with a $C^1$ map $\phi:U\to \mathbb R^n$ where $U\in \mathbb R^n$ is open. Let $E\subset U$ be compact. 
For $x \in \bbR^{n}$, we write $d(x)$ to denote the distance from $x$ to $E$. We write $\dia E$ to denote the diameter of $E$ and ${\rm card}(E)$ to denote the cardinality of $E$. Let $\ep>0$ be given. We make the following
\phantomsection
\subsection*{\textsc{\ul{Assumptions:}}}
\addcontentsline{toc}{subsection}{Assumptions}
\begin{enumerate}[leftmargin=*, label={(\thesection.\arabic*)}]
\item \label{i:0.2} (Geometry of E). For certain positive constants $c_{0}$, $C_{1}$, $c_{2}$ depending on $n$ such that the following holds, \\
Let $x \in \bbR^{n} \backslash E$. If $d(x) \leq c_{0} \, \dia E$, then there exists a ball $B(z,r) \subset E$ such that $\ml z-x \mr \leq C_{1} \, d(x)$ and $r \geq c_{2} \, d(x)$.
\item \label{i:0.3} (Small $\ep$ assumption). $\ep$ is less than a small enough constant determined by $c_{0}$, $C_{1}$, $c_{2}$, $n$. 
\item \label{i:0.1} (Geometry of $\phi$). For $x,y \in E$ we have $(1-\ep) \ml x-y \mr \leq \ml \phi(x) - \phi(y) \mr \leq (1+\ep) \ml x-y \mr$.
\end{enumerate}

Henceforth, by a Euclidean motion, we shall mean a map $A:x\to Tx+x_0$ from $\mathbb R^n$ to $\mathbb R^n$ with $T\in O(n)$ or 
$T\in SO(n)$ and $x_0\in \Bbb R^n$ fixed. $O(n)$ and $SO(n)$ will denote the orthogonal and special orthogonal groups respectively. A 
Euclidean motion $A$ will be called "proper" if $T\in SO(n)$ otherwise it is called improper. $A_{\infty}$ will always denote the identity Euclidean motion.
\footnote{$O(n)$ is by definition the group of isometries of $\mathbb R^n$ which preserve a fix point and $SO(n)$ the subgroup of $O(n)$ of orthogonal matrices of determinant 1. The Euclidean motions of $\mathbb R^n$ are the elements of the symmetry group of $\mathbb R^n$, ie 
all isometries of $\mathbb R^n$.} 

We will prove as our main result:

\begin{thm}
	\label{t:0.1}
	\addcontentsline{toc}{subsection}{Theorem 0.1}
	Under the above assumptions, there exists a $C^{1}$ map $\Phi: \bbR^{n} \to \bbR^{n}$ and a Euclidean motion $A_{\infty}: \bbR^{n} \to \bbR^{n}$, with the following properties,
\end{thm}
\vspace{-1.8em}
\begin{enumerate}[leftmargin=*, label={(\thesection.\arabic*)}]
\setcounter{enumi}{3}
\item \label{i:0.4} \textit{$(1-C\ep)\ml x-y \mr \leq \ml \Phi(x) - \Phi(y) \mr \leq (1+C\ep)\ml x-y \mr$ for all $x,y \in \bbR^{n}$; here $C$ is determined by $c_{0}$, $C_{1}$, $c_{2}$, $n$}.

\item \label{i:0.5} \textit{$\Phi = \phi$ in a neighborhood of $E$}.

\item \label{i:0.6} \textit{$\Phi = A_{\infty}$ outside $\bl x \in \bbR^{n}: d(x) < c_{0} \br$}.

\item \label{i:0.7} \textit{$\Phi: \bbR^{n} \to \bbR^{n}$ is one-to-one and onto}.

\item \label{i:0.8} \textit{If $\phi \in C^{m}(U)$ for some given $m \geq 1$, then $\Phi \in C^{m}\(\bbR^{n}\)$}.

\item \label{i:0.9} \textit{If $\phi \in C^{\infty}(U)$, then $\Phi \in C^{\infty}(\bbR^{n})$}.
\end{enumerate}

\subsection{Our Plan} Theorem~\ref{t:0.1} comes with an interesting tale whose foundations are in the work of \cite{DF4, FD1, FD3}.  We set out now to describe this tale in detail which will then allow us to explain Theorem~\ref{t:0.1} and place it in context. Once we have done this, will then devote our attention to proving Theorem~\ref{t:0.1}. As part of our proof, we will need machinary from \cite{DF4, FD1, FD3} and from \cite{F1,F3,F4,F5,F6} and we will discuss this machinary and place it in context when we need it. 

We begin with:

\section{A tale of Whitney distorted extensions, interpolation and alignment in $\mathbb R^n$ for fine sets-I} 
\setcounter{thm}{0}
\setcounter{equation}{0}
\subsection{$\varepsilon'$-distorted diffeomorphisms}
In order to begin our study of Theorem~\ref{t:0.1}, we begin by looking at the case when the set $E$ is finite. Here and throughout $\ep'>0$ will be a small enough positive number which depends on $n$.  We require 
a class of smooth 1-1 and onto small distortions from $\mathbb R^n \to\mathbb R^n$ intoduced first in \cite{FD3}
which we call $\varepsilon'$-distorted diffeomorphisms. 
\begin{dfn}
A diffeomorphism (hence 1-1 and onto) $\Phi:\mathbb R^n\to \mathbb R^n$ is ``$\varepsilon$'-distorted" provided
\beq
(1+\varepsilon')^{-1}I\leq [\nabla \Phi(x)]^{T}[\nabla \Phi (x)]\leq (1+\varepsilon')I
\label{e:edistortion}
\eeq
as matrices, for all $x\in \mathbb R^n$. Here, $I$ denotes the identity matrix in $\mathbb R^n$.  Henceforth all $\varepsilon'$-distorted diffeomorphisims will be understood from $\mathbb R^n$ to $\mathbb R^n$.
\end{dfn}

If $\Phi$ is $\varepsilon'$-distorted then an application of Bochner's theorem gives that we have for all $x,y\in \mathbb R^n$,
\beq
(1+\varepsilon')^{-1}|x-y|\leq |\Phi(x)-\Phi(y)|\leq |x-y|(1+\varepsilon').
\label{e:edistortion1}
\eeq

\footnote{$\Phi$ is bi-lLipchitz with constants $(1+\varepsilon')^{-1}$ and $(1+\varepsilon')$ and so it is a small distortion.} A $\varepsilon'$-distorted diffeomorphism $\Phi$ is proper if ${\rm det}(\Phi')>0$ on $\mathbb R^n$ and improper if ${\rm det}(\Phi')<0$
on $\mathbb R^n$. Since ${\rm det}(\Phi')\neq 0$ everywhere on $\mathbb R^n$, every $\varepsilon'$- distorted
diffeomorphism is proper or improper. $\varepsilon$-distorted diffeomorphisms are defined similarly. 
\subsection{A Whitney distortion extension theorem: The main result of \cite{FD3}}
In \cite{FD3}, we proved the following Whitney distortion extension theorem:

\begin{thm}
Let $\varepsilon'>0$ and let $y_1,...y_k$ and $z_1,...z_k$ be two $1\leq k\leq n$ sets of distinct points in $\mathbb R^n$. Then there exists $\beta'>0$ depending only on $\varepsilon'$ such that the following holds: Suppose that
\beq
(1+\beta')^{-1}\leq \frac{|z_i-z_j|}{|y_i-y_j|}\leq (1+\beta'),\, 1\leq i,j\leq k,\, i\neq j.
\label{e:emotionsa}
\eeq
Then there exists a $\varepsilon'$-distorted diffeomorphism $\Phi$ satisfying
\beq
\Phi(y_i)=z_i,\, 1\leq i\leq k.
\label{e:emotionsaaa}
\eeq
\label{t:lemmafive}
\end{thm}

\begin{itemize}
\item[(a)] Note that Theorem~\ref{t:lemmafive} works for any two distinct $1\leq k\leq n$ configurations $y_1,...,y_k$, $z_1,...,z_k$ satisfying the condition (\ref{e:emotionsa}). If we assume that there exists a map $\phi:E\to \mathbb R^n$ 
where $E:=\left\{y_1,y_2,...,y_k\right\}$ and $\phi(E):=\left\{z_1,z_2,...,z_k\right\}$, then Theorem~\ref{t:lemmafive} is a {\bf Whitney extension theorem in $\mathbb R^n$}. In particular, it is a generalization of a well known result of John on extensions of isometries. See for example \cite{WW1}. \footnote{Note that we are extending a bi-Lipchitz diffeomorphism which may not be $C^1$.}
\end{itemize}

\subsection{The intriguing restriction in Theorem~\ref{t:lemmafive}, ${\rm card}(E)=k\leq n$ with $n$ the dimension of $\mathbb R^n$. }
Theorem~\ref{t:lemmafive} has an intriguing feature, namely the restriction ${\rm card}(E)=k\leq n$ with $n$ the dimension of $\mathbb R^n$. Indeed, we showed in \cite{FD3} that Theorem~\ref{t:lemmafive} is false if $k>n$ by constructing a counter example for $n\geq 2$.
In the work of  \cite{FD1}, we showed how to remove this restriction on the number of points $k$ if 
roughly we require that on any $n+1$ of the $k$ points which form a relatively voluminous simplex, the extension $\Phi$ is orientation preserving.
We will discuss these results below in more detail. 

\subsection{The constants $\varepsilon$ and $\beta'$ and how they relate to each other.}
Notice that Theorem~\ref{t:lemmafive} has two positive constants $\varepsilon'$ and $\beta'$. $\varepsilon'$ determines the distortion of the map $\Phi$ whereas $\beta'$ determines the distortion of the pairwise distances between point sets $y_1,...,y_k$ and $z_1,...,z_k$ and depends only on $\varepsilon'$.
When looking at Theorem~\ref{t:lemmafive}, an immediate question which comes to mind is what is the quantative dependence of $\beta'$ on $\varepsilon'$. In \cite{FD1}, we proved:

\begin{thm}
Let $\varepsilon'>0$ and let $y_1,...y_k$ and $z_1,...z_k$ be two $1\leq k\leq n$ sets of distinct points in $\mathbb R^n$ with
\beq
\sum_{i\neq j}|y_i-y_j|^2+\sum_{i\neq j}|z_i-z_j|^2=1,\, y_1=z_1=0.
\label{e:lemmafourpoints}
\eeq
Then there exist constants $J,J'>0$ depending only on $n$  such the following holds: Set $\beta'=J'\varepsilon'^{J}$ and suppose 
\beq
||z_i-z_j|-|y_i-y_j||<\beta', \, i\neq j.
\label{e:lemmafourpointsa}
\eeq
Then there exists a $\varepsilon'$-distorted diffeomorphism $\Phi$ satisfying
\beq
\Phi(y_i)=z_i,\, 1\leq i\leq k.
\label{e:emotionsaaa}
\eeq
\label{t:lemmafiveprime}
\end{thm}

The main results of \cite{FD3}  below show under suitable conditions on $E$  that one can take typically $\beta'=\exp(-C_{K'}/\varepsilon)$ with a positive constant $C_{K'}$ depending on 
$n$ and a special fixed $K'\geq 1$ which depends on $n$.  Clearly  $\beta'=\exp(-C_{K'}/\varepsilon)$ is a more refined estimate than $\beta'=J'\varepsilon'^{J}$.  $\beta'=C\varepsilon$ for a $C>0$ depending on $n$  is clearly optimal and our main result in this paper Theorem~\ref{t:0.1} achieves this for certain $E$ and $\phi$.  

\subsection{Theorem~\ref{t:lemmafive}, Theorem~\ref{t:lemmafiveprime}: Interpolation in $\mathbb R^n$-I.} 
Notice that Theorem~\ref{t:lemmafive} and Theorem~\ref{t:lemmafiveprime} are {\bf Interpolation of data results in $\mathbb R^n$} with the map $\Phi$ the interpolation operator. 
If we assume that there exists a map $\phi:U\to \mathbb R^n$, $U\mathbb R^n$ open with  
$E:=\left\{y_1,y_2,...,y_k\right\}\subset U$ and $\phi(E):=\left\{z_1,z_2,...,z_k\right\}$ for which we have an extension $\Phi$ defined over all $\mathbb R^n$, then it is natural to want to estimate a suitable norm of $\Phi$ so that one may be able to answer questions of how well $\Phi$ approximates $\phi$ in $U\setminus E$.  Extensions $\Phi$ and sets $E$ which achieve good and optimal approximations are interesting questions and we address this problem in a forthcoming paper. 

We refer the reader to the papers by Lubinsky and his collaborators and the many references cited
therein for a good insight into these questions for  interpolation by polynomials \cite{L1,L2,L3,L4,L5,L6,L7,DL,DL1}. The sizes and optimal sizes of interpolation/extension norms and the operators which achieve these in Whitney extension problems is a difficult problem in general and has been addressed substantially by Fefferman and his collaborators in \cite{F1,F3,F4,F5,F6,KF,KF1,KF2,KF3, FIL1,FIL2,FIL3,FIL4,FIL5,FMH,I} and several other papers referenced therein.
\medskip

We have explained how  Theorem~\ref{t:lemmafive} and Theorem~\ref{t:lemmafiveprime} are Whitney distortion extensions and interpolation results in $\mathbb R^n$. 
We now show they are results on alignment of data in $\mathbb R^n$ as well. To understand this, we begin with:

\subsection{Alignment of data in Euclidean space and the Procrustes problem.}
An important problem in computer vision is comparing $k\geq 1$ point configurations in $\Bbb R^n$. \footnote{In computer vision, the phrase "$k\geq 1$ point configuration" means $k\geq 1$ distinct points.} One way to think of this is as follows: Given two $k$ configurations in $\Bbb R^n$, do there exist combinations of rotations, translations, reflections and compositions of these which map the one configuration onto the other. This is the shape registration problem. \footnote{In computer vision, two point clouds are said to have the same ``shape" if there exist combinations of rotations, translations, reflections and compositions of these which map the one set of points onto the other. This typically is called the registration problem.} A typical application of this problem arises in image processing and surface alignment where it is often necessary to align an unidentified image to an image from a given data base of images for example parts of the human body in face or fingerprint recognition. Thus the idea is to recognize points (often called landmarks) by verifying whether they align to points in the given data base. 
An image in $\Bbb R^{n}$ does not change under Euclidean motions. 
Motivated by this, in this paper, we will think of shape preservation in terms of whether there exists a Euclidean motion which maps one $k$ point configuration onto a second. \footnote{The shape identification problem can be stated more generally for other transformation groups. 
We restrict ourselves in this paper to Euclidean motions.}
In the case of labelled data (where the data points in each set are indexed by the same index set), an old approach called the Procrustes approach \cite{G,G1} analytically determines a Euclidean motion which maps the first configuration close to the other (in a $L^2$) sense.
There are a variety of ways to label points. See for example \cite{LW,WW}. 
In terms of good algorithms to do alignment of this kind, 
one method of Iterative Closet Point (ICP) for example is very popular and analytically computes an optimal rotation for alignment of $k$ point configurations, see for example \cite{DL} and the references cited therein for a perspective. 
Researchers in geometric processing think of the problem of comparing point clouds or finding a distance between them as to asking how to deform one point cloud into the other (each point cloud represented by a collection of meshes) in the sense of saying they have the same shape. 
We refer the reader to the references \cite{P,LW,WW,BCSZ,BSAB,WSW,WS,ZS,SW,DL,LAD,LD1,LD2, L, L1, BLCPFPJD, L2, SWK, SSK, VLBRC, LRS} and the many references cited therein for a broad perspective on applications of this problem.
\medskip

One way to dig deeper into the Procrustes problem is to compare pairwise distances between labelled points. In this regard, the following result is well known. See for example \cite{ATV,WW1}.
\medskip

\begin{thm}
Let $y_1,...,y_k$ and $z_1,...,z_k$ be two $k\geq 1$ point configurations in $\mathbb R^n$. Suppose that
\[
|z_i-z_j|=|y_i-y_j|,\, 1\leq i,j\leq k,\, i\neq j.
\]
Then there exists a Euclidean motion $A$ such that $A(y_i)=z_i,\, i=1,...,k.$ If $k\leq n$, then $A$ can be taken as proper.
\label{t:Theorem 1}
\end{thm}

Alignment of point configurations from their pairwise distances are encountered for example in X-ray crystallography and in the mapping of restriction sites of DNA. See \cite{P,SSL} and the references cited therein. (In the case of one dimension, this problem is known as the turnpike problem or in molecular biology, it is known as the partial digest problem). See the work of \cite{RS,LW} for example which deals with algorithms and their running time for such alignments. We mention that a difficulty in trying to match point configurations is the absence of labels in the sense that often one does not know which point to map to which. We will not deal with the unlabeled problem in this paper but see \cite{NDS}. We refer the reader to the papers by Werman and his 
collaborators and the many references cited therein for a good insight into the subject of alignment of data in Euclidean space. 
\cite{WW,DSW,PW,AKMSW,WO,BW,OW}.

\subsection{ Labelled approximate alignment in $\mathbb R^n$ in the case where pairwise distances are distorted.}
To study the labelled alignment data problem in the case where pairwise distances are distorted, we proved in \cite{FD3} an analogy of Theorem~\ref{t:Theorem 1} namely Theorem~\ref{t:lemmafour}:
\medskip

\begin{thm}\begin{itemize}\item[(a)] Given $\varepsilon'>0$, there exists $\beta'>0$ depending only on $\varepsilon'$, such that the following holds. Let $y_1,...,y_k$ and
$z_1,...,z_k$ be two $k\geq 1$ point configurations in $\mathbb R^n$ satisfying $(\ref{e:emotionsa})$.
Then, there exists a Euclidean motion $A$ such that
\beq
|z_i-A(y_i)|\leq \varepsilon' {\rm diam}\left\{y_1,...,y_k\right\}
\label{e:emotionsb}
\eeq
for each $1\leq i\leq k$. If $k\leq n$, then we can take $A$ to be a proper. 
\item[(b)] Suppose now that $y_1,...,y_k$ and
$z_1,...,z_k$ are two $k\geq 1$ point configurations in $\mathbb R^n$ so that $(\ref{e:lemmafourpoints})$ holds.
Then there exist constants $J,J'>0$ depending on $n$ such the following holds: Set $\beta'=J'\varepsilon'^{J}$ and suppose $(\ref{e:lemmafourpointsa})$.
Then, there exists a Euclidean motion $A$ such that
\beq
|z_i-A(y_i)|\leq \varepsilon'. 
\label{e:emotionsu}
\eeq
\label{t:lemmafour}
\end{itemize}
\end{thm}

Notice that if we require each set of points $y_1,...,y_k$ and $z_1,...,z_k$ to be in a bounded set of controlled radius, then we are able to be specific about the relationship between $\varepsilon'$ and $\beta'$ in Theorem~\ref{t:lemmafour},
namely $\beta'=J'\varepsilon'^{J}$. We provide the proof of Theorem~\ref{t:lemmafour} from \cite{FD3} in order to illustrate in particular, regarding (b), the use of an inequality called Lojasiewicz's inequality which is a technique allowing to approximate points in $E$ by certain varieties in $\mathbb R^n$.
A similar idea can be found for example in \cite{O,OST}.

\subsection{ Lojasiewicz's inequality and Proof of Theorem~\ref{t:lemmafour}.}
{\bf Proof of Theorem~\ref{t:lemmafour}} \ We first prove (a): Suppose not. Then for each $l\geq 1$, we can find points $y_1^{(l)},...,y_k^{(l)}$ and $z_1^{(l)},...,z_k^{(l)}$ in
$\mathbb R^D$ satisfying (\ref{e:emotionsa}) with $\beta'=1/l$ but not satisfying (\ref{e:emotionsb}). Without loss of generality, we may suppose that ${\rm diam}\left\{y_1^{(l)},...,y_k^{(l)}\right\}=1$ for each $l$ and that $y_1^{(l)}=0$ and
$z_l^{(1)}=0$ for each $l$. Thus $|y_i^{(l)}|\leq 1$ for all $i$ and $l$ and
\[
(1+1/l)^{-1}\leq \frac{|z_i^{(l)}-z_j^{(l)}|}{|y_i^{(l)}-y_j^{(l)}|}\leq (1+1/l)
\]
for $i\neq j$ and any $l$.
However, for each $l$, there does not exist an Euclidean motion
$A$ such that
\beq
|z_i^{(l)}-A(y_i^{(l)})|\leq \varepsilon'
\label{e:emotionsc}
\eeq
for each $i$. Passing to a subsequence, $l_1,l_2,l_3,...,$ we may assume
\[
y_i^{(l_{\mu})}\to y_i^{\infty},\, \mu\to \infty
\]
and
\[
z_i^{(l_{\mu})}\to z_i^{\infty},\, \mu\to \infty.
\]
Here, the points $y_i^{\infty}$ and $z_i^{\infty}$ satisfy
\[
|z_i^{\infty}-z_j^{\infty}|=|y_i^{\infty}-y_j^{\infty}|
\]
for $i\neq j$. Hence, by Theorem~\ref{t:Theorem 1}, there is an Euclidean motion $A^*$ such that $A*(y_i^{\infty})=z_i^{\infty}$. Consequently,
for $\mu$ large enough, (\ref{e:emotionsc}) holds with $l_{\mu}$. This contradicts the fact that for each $l$, there does not exist a $A$ satisfying (\ref{e:emotionsc}) with $l$.
Thus, we have proved all the assertions of (a) except that we can take $A$ to be proper if $k\leq n$. To see this, suppose that $k\leq n$ and let $A$ be an improper Euclidean motion such that
\[
|z_i-A(y_i)|\leq \varepsilon'{\rm diam}\left\{y_1,...,y_k\right\}
\]
for each $i$. Then, there exists an improper Euclidean motion $A^*$ that fixes $y_1,...,y_k$ and in place of $A$, we may use the map $A^* o A$ in (a) so (a) is proved. We now prove (b). Here we use an inequality
called Lojasiewicz's inequality, \cite{SJS} which allows us to control the upper bound estimate in (\ref{e:emotionsb}) and replace it by the upper bound in (\ref{e:emotionsu}) provided the points $y_1,...,y_k$ and $z_1,...,z_k$ each lie in bounded sets of controlled radius.
The Lojasiewicz's inequality says the following: Let $f:U\to \Bbb R$ be a real analytic function on an open set $U$ in $\Bbb R^n$ and $Z$ be the zero set of $f$. Assume that $Z$ is not empty. 
Then for a compact set $K$ in $U$, there exist positive constants $J$ and $J'$ depending on $f$ and $K$ such that uniformly for all $x\in K$, $|x-Z|^{J}\leq J'|f(x)|$.
It is easy to see that using this, one may construct approximating distinct points $y_1',...,y_k',z_1',...,z_k'\in \mathbb R^n$ (zeroes of a suitable $f$) with the following two properties:
(1) There exist positive constants $J,J'>0$ depending on $n$ such that 
\[
|(y_1,...,y_k,z_1,...,z_k)-(y_1',...,y_k',z_1',...,z_k')| \leq
J\varepsilon'^{J'}.
\]
In particular, we have
\[
|y_i-y_i'|\leq J\varepsilon'^{J'}
\]
and
\[
|z_i-z_i'|\leq J\varepsilon'^{J'}.
\]
(2) $|y_i'-y_j'|=|z_i'-z_j'|$ for every $i,j$.
Thanks to (2), we may choose a Euclidean motion $A$ so that $A(y^\prime_i)= z^\prime_i$ for each $i$.
Also, thanks to (1), there exists positive constants $J_1, J_2$ depending on $n$ with
\[
|A(y_i)-A(y_i')|\leq J_1\varepsilon'^{J_2}
\]
So it follows that there exist positive constants $J_3, J_4$ depending only on $n$ with
\[
|A(y_i)-z_i|\leq J_3\varepsilon'^{J_4}
\]
which is (b). $\Box$

\section{Building $\varepsilon'$-distorted diffeomorphisms; Slow twists and Slides, approximation of Euclidean motions.}
\setcounter{thm}{0}
\setcounter{equation}{0}

\subsection{Introduction to the ideas behind alligment and Theorem~\ref{t:lemmafive} and Theorem~\ref{t:lemmafiveprime}.}
In \cite{FD3} (see Example~\ref{e:Example1} and Example~\ref{e:Example2} below), we introduced certain slow rotations and translations as $\varepsilon'$
distorted diffeomorphisms. We call them {\it Slow twists} and {\it Slides}.
The $\varepsilon'$-distortions in  Theorem~\ref{t:lemmafive} and Theorem~\ref{t:lemmafiveprime} as is shown in \cite{FD3}, are  built using Slow twists and Slides and so 
Theorem~\ref{t:lemmafive} and Theorem~\ref{t:lemmafiveprime} are also an {\bf Alignment of data result} in Euclidean space.  Slow twists are $\varepsilon'$-distorted diffeomorphisms whose argument is a function of distance from the origin. These rotations reduce their speed of rotation for decreasing argument, are non-rigid for decreasing argument becoming rigid for increasing argument. See (\ref{e:lemmaone}). Slides are translations which are $\varepsilon'$-distorted diffeomorphisms and satisfy (\ref{e:lemmatwo}). (See \cite{DSW,NDS} for applications of Slow Twists and Slides to data allignment). 
\footnote{This a natural idea given we recall that
$O(n)$ is by definition the group of isometries of $\mathbb R^n\to\mathbb R^n$ which perserve a fix point ($SO(n)$ is the subgroup of $O(n)$ of orthogonal matrices of determinant 1) and Euclidean motions from 
$\mathbb R^n\to \mathbb R^n$ are the elements of the symmetry group of $\mathbb R^n$, ie all isometries of $\mathbb R^n$.}
In essence, $\varepsilon'$-distorted diffeomorphisms approximate Euclidean motions well. 
\medskip
It is instructive to provide more detail to these ideas:
 \subsection{Slow twists.}

\begin{exm}
{\rm Let $\varepsilon'>0$ and $x\in \mathbb R^n$. Let $S(x)$ be the $D\times D$ block-diagonal matrix
\[
\left(
\begin{array}{llllll}
D_1(x) & 0 & 0 & 0 & 0 & 0 \\
0 & D_2(x) & 0 & 0 & 0 & 0 \\
0 & 0 & . & 0 & 0 & 0 \\
0 & 0 & 0 & . & 0 & 0 \\
0 & 0 & 0 & 0 & . & 0 \\
0 & 0 & 0 & 0 & 0 & D_r(x)
\end{array}
\right)
\]
where for each $i$ either $D_i(x)$ is the $1\times 1$ identity matrix or else
\[
D_i(x)=\left(
\begin{array}{ll}
\cos f_i(|x|) & \sin f_i(|x|) \\
-\sin f_i(|x|) & \cos f_i(|x|)
\end{array}
\right)
\]
where $f_i:\mathbb R\to \mathbb R$ are functions satisfying the condition: $t|f_i'(t)|<J_1\varepsilon'$  for some constant $J_1>0$ depending only on $n$, uniformly for $t\geq 0$. 
The $1\times 1$ identity matrix is used to compensate for the even/odd size of the matrix. Let 
$\Phi(x)=\Theta^{T}S(\Theta x)$ where $\Theta$
is any fixed matrix in $SO(n)$. Then $\Phi$ is a $\varepsilon'$-distorted diffeomorphism and we call it a {\it slow twist} (in analogy to rotations). }
\label{e:Example1}
\end{exm}

\subsection{Slides.}
\begin{exm}
{\rm Let $\varepsilon'>0$ and let $g:\mathbb R^{n}\to \mathbb R^{n}$
be a smooth map such that $|g'(t)|<J_2\varepsilon'$ for some constant $J_2>0$ depending only on $n$, uniformly for $t\in \mathbb R^n$. 
Consider the map $\Phi(t)=t+g(t)$, $t\in \mathbb R^n$. Then $\Phi$ is a $\varepsilon'$-distorted diffeomorphism. We call the map $\Phi$ a {\it slide} (in analogy to translations). }
\label{e:Example2}
\end{exm}

Using the definitions of Slow Twists and Slides we are able to now build $\varepsilon'$-distorted diffeomorphisms and approximate Euclidean motions well by them.  This is given in
Theorem~\ref{t:lemmatwists} below.

\subsection{Building $\varepsilon'$-distorted diffeomorphisms and approximation of Euclidean motions.}
We have:

\begin{thm}
Let $\varepsilon'>0$. 
\begin{itemize}
\item[(a)] There exists $\eta>0$ depending on $\varepsilon'$ for which the following holds. Let $\Theta\in SO(n)$, $r_1,r_2>0$ and let $0<r_1\leq \eta r_2$.
Then, there exists an $\varepsilon'$-distorted diffeomorphism $\Phi$ such that
\beq
\left\{
\begin{array}{ll}
\Phi(x)=\Theta x, & |x|\leq r_1 \\
\Phi(x)=x, & |x|\geq r_2
\end{array}
\right.
\label{e:lemmaone}
\eeq
\item[(b)] There exists $\eta_1>0$ depending on $\varepsilon'$ such that the following holds. Let $A$ be a proper Euclidean
motion. Suppose $0<r_3\leq \eta_{1} r_4$ and $|x_0|\leq c\varepsilon' r_3$.
Then there exists an $\varepsilon'$-distorted diffeomorphism $\Phi_1$ such that
\beq
\left \{
\begin{array}{ll}
\Phi_{1}(x)=A(x), & |x|\leq r_3 \\
\Phi_{1}(x)=x, & |x|\geq r_4
\end{array}
\right.
\label{e:lemmatwo}
\eeq
\item[(c)] There exists $\eta_2>0$ depending on $\varepsilon'$ such that the following holds. Let $r_5,r_6>0$ with $0<r_5\leq \eta_2 r_6$ and let $x,x'\in \mathbb R^{D}$ with $|x-x'|\leq c\varepsilon' r_5$ and 
$|x|\leq r_5$. Then, there exists an $\varepsilon'$-distorted diffeomorphism $\Phi$ such that 
\beq
\Phi(x)=x'\, {\rm and}\, \Phi(y)=y,\, {\rm for}\, |y|\geq r_6.
\label{e:corollaryone}
\eeq
\item[(d)] Let $r>0$, $x_1\in \mathbb R^D$ and let $B(x_1,r)$ be a ball in $\mathbb R^n$. Let $A$ and $A^*$ be proper Euclidean motions such that
\beq
|A(x_1)-A^*(x_1)|\leq \varepsilon' r.
\label{e:emotionsuu}
\eeq
Then there exists a $C\varepsilon'$-distorted diffeomorphism $\Phi$ such that $\Phi=A$ in
$B(x_1,\exp(-1/\varepsilon')r)$ and $\Phi=A^*$ outside $B(x_1, r)$. Here $C>0$ depends only on $n$.
\end{itemize}
\label{t:lemmatwists}
\end{thm}

 Theorem~\ref{t:lemmatwists} together with Theorem~\ref{t:lemmafour} now form the basis for the proof of Theorem~\ref{t:lemmafive} and Theorem~\ref{t:lemmafiveprime}.
\medskip

We are ready to continue our tale: 

\section{A tale of Whitney distorted extensions, interpolation and alignment in $\mathbb R^n$ for fine sets-II.}
\setcounter{thm}{0}
\setcounter{equation}{0}

We recall that Theorem~\ref{t:lemmafive} and Theorem~\ref{t:lemmafiveprime} have an intriguing feature, namely the restriction of the number of points $k$ to be bounded by the dimension of $\mathbb R^n$, $n$. Indeed, we showed in \cite{FD3} that Theorem~\ref{t:lemmafive} is false if $k>n$ by constructing a counter example for $n\geq 2$.
Theorem~\ref{t:Theorem2a}, Theorem~\ref{t:Theorem2b}, Theorem~\ref{t:Theorem3}, 
Theorem~\ref{t:Theorem4} and Theorem~\ref{t:cextensionblock2}, the main results of \cite{FD1} which we will describe below, tell us that we may remove the restriction on $k$ if 
roughly we require that on any $n+1$ of the $k$ points which form a relatively voluminous simplex, the extension $\Phi$ is orientation preserving. We will make this notation more precise by meeting a special object called an $\eta$ block. Throughout this section $n\geq 2$.
\subsection{$\eta$ block} 

{\begin{dfn}
{\rm For $z_0,z_1,...,z_l \in \mathbb R^n$ with $l\leq n$, $V_l(z_0,...,z_l)$ will denote the $l$-dimensional volume of the $l$-simplex with vertices at $z_0,...,z_l$. If $E\subset \mathbb R^n$ is a finite set, then $V_l(E)$ denotes the max of $V_l(z_0,...,z_l)$ over all 
$z_0,z_1,...,z_l\in S$. \footnote{If $V_n(E)$ is small, then we expect that $E$ will be close to a hyperplane in $\mathbb R^n$.} Let 
$\phi:E\to \mathbb R^n$ and let $0<\eta<1$. A positive (resp. negative) $\eta$-block for $\phi$ is a $n+1$ tuple $(x_0,...,x_n)\in \mathbb R^n$ such that the following two conditions hold: (1) $V_n(x_0,...,x_n)\geq (\leq) \eta^n{\rm diam}(x_0,...,x_n)$. (2) Let $T$ be the unique affine map which agrees with $\phi$ on $E$. $T$ exists and is unique by virtue of \cite{ATV}. Then we assume that $T$ is proper or improper. (Here we recall that an invertible affine map $T':\mathbb R^n\to \mathbb R^n$ is proper if ${\rm det}(T')>0$ and improper if ${\rm det}(T')<0$. Since $T'$ is invertible, $T'$ is either proper or improper.) Thus if the map $T$ in (2), is not invertible then $(x_0,...,x_n)$ is not an 
$\eta$ block.) }
\label{d:block}
\end{dfn}

We are ready for: 
\subsection{ Five results.}

\begin{thm}
Let $E\subset \mathbb R^n$ be finite. There exists a positve constant $K'\geq 1$ depending on $n$ and positive constants $c_{K'}$, $C'_{K'}$, $C''_{K'}$ depending on $n$ and $K'$ such that the 
following holds: Set $\eta=\exp(-C'_{K'}/\varepsilon)$ and $\beta=\exp(-C''_{K'}/\varepsilon)$ with $0<\varepsilon<c_{K'}$.  Let $\phi:E\to \mathbb R^n$ satisfy
\beq
(1+\beta)^{-1}|x-y|\leq |\phi(x)-\phi(y)|\leq (1+\beta)|x-y|, \, x, y\in E.
\label{e:deltadistortion}
\eeq
Then if $\phi$ has no negative $\eta$ block, there exists a proper $\varepsilon$-distorted diffeomorphism
$\Phi$ such that $\phi=\Phi$ on $E$ and $\Phi$ agrees with the proper Euclidean motion $A_{\infty}$ on
\[
\left\{x\in \mathbb R^n:\, {\rm dist}(x,E)\geq 10^4{\rm diam}(E)\right\}.
\]
\label{t:Theorem2a}
\end{thm}

\begin{thm}
Let $E\subset \mathbb R^n$ be finite. There exists a positve constant $K'\geq 1$ depending on $n$  and positive constants $c_{K'}$, $C'_{K'}$, $C''_{K'}$ depending only on $n$ and $K'$ such that the following holds: Set $\eta=\exp(-C'_{K'}/\varepsilon)$ and $\beta=\exp(-C''_{K'}/\varepsilon)$ with $0<\varepsilon<c_{K'}$. Let $\phi:E\to \mathbb R^n$ satisfy
$(\ref{e:deltadistortion})$. Then if $\phi$ has a negative $\eta$ block, $\phi$ cannot be extended to a proper $\beta$ distorted diffeomorphism. 
\label{t:Theorem2b}
\end{thm}

\begin{thm}
Let  $E\subset \mathbb R^n$ be finite. There exists a positve constant $K'\geq 1$ depending on $n$ and positive constants $c_{K}$, $C'_{K'}$ depending only on $n$ and $K'$ such that the 
following holds: Set $\beta=\exp(-C'_{K'}/\varepsilon)$ with $0<\varepsilon<c_{K'}$.  Let $\phi:E\to \mathbb R^n$ satisfy
$(\ref{e:deltadistortion})$. Suppose that for any $S_o\subset S$ with at most $2n+2$ points, there exists a $\beta$ distorted diffeomorphism $\Phi^{S_0}$  such that $\Phi^{S_0}=\phi$ on $S_0$. Then, there exists an $\varepsilon$-distorted diffeomorphism $\Phi$ such that $\Phi=\phi$.
\label{t:Theorem3}
\end{thm}

\begin{thm}
Let $E\subset \mathbb R^n $ with ${\rm card}(S)\leq n+1$. There exist positive constants $c,C$ depending only on $n$ such that the following holds: Set 
$\beta=\exp(-C/\varepsilon)$ with $0<\varepsilon<c$ and let $\phi:S\to \mathbb R^n$ satisfy
$(\ref{e:deltadistortion})$.
Then there exists a $\beta$-distorted diffeomorphism $\Phi:\mathbb R^n\to \mathbb R^n$ such that $\Phi=\phi$.
\label{t:Theorem4}
\end{thm}

\begin{thm}
Let $\phi:E\to \mathbb R^n$ where $E\in \mathbb R^n$ is finite and let $0<\eta<1$. Suppose that $\phi$ satisfies $(\ref{e:deltadistortion})$ and has a positive $\eta$ block and a negative
$\eta$ block. Let $0<\beta<c\eta^n$ for small enough $c>0$ depending only on $n$. Then $\phi$ does not extend to a
$\beta$ distorted diffeomorphism $\Phi:\mathbb R^n\to \mathbb R^n$.
\label{t:cextensionblock2}
\end{thm}

Note that in Theorem~\ref{t:Theorem2a}, Theorem~\ref{t:Theorem2b}, Theorem~\ref{t:Theorem3}, 
Theorem~\ref{t:Theorem4} and Theorem~\ref{t:cextensionblock2}, we are also able to improve the order $\beta$ a polynomial in $\varepsilon$  from to $\beta$ to 
$\exp(-\frac{C_K}{\varepsilon}$ for a constant $C_K>0$ dependning on $n$ and $E$.

\section{Approximation by Euclidean Motions.}
\setcounter{equation}{0}
We now want an analogy of Theorem~\ref{t:lemmareflection3} for a proper $\varepsilon$- diffeomorphisms (see Theorem~\ref{t:lemmareflectionex4} below) and it is here that we meet the constant $K'$. See (\ref{e:specialk}) 
below.
In order to do this, we need more machinery. To begin, it will be necessary first
to study pointwise approximation of $\varepsilon$ diffeomorphisms by given Euclidean motions. It follows from \cite{DF4} 
that given $\varepsilon''>0$ (small enough and depending only on $D$) and $\Phi:\mathbb R^D\to \mathbb R^D$ a $\varepsilon''$-distorted diffeomorphism, there exists a Euclidean motion $A:\mathbb R^D\to \mathbb R^D$ with 
$|\Phi(x)-A(x)|\leq C\varepsilon''$ for $x\in B(0,10)$. 
Actually, using the well-known John-Nirenberg inequality, in \cite{DF4} we proved a lot more, namely a BMO theorem for $\varepsilon''$-distorted diffeomorphisms which is in the next subsection. 

\subsection{BMO theorem for $\varepsilon'$-distorted diffeomorphisms.} 
{\bf BMO theorem for $\varepsilon''$-distorted diffeomorphisms}:\, Let $\varepsilon''>0$ be a small enough positive number depending only on $D$,  $\Phi:\mathbb R^D\to \mathbb R^D$  a $\varepsilon''$-distorted diffeomorphism and let $B\in \mathbb R^D$ be a ball. There exists $T=T_B\in O(D)$ and 
$C>0$ such that for all $\lambda\geq 1$, we have
\beq
{\rm vol}\left\{x\in B:|\Phi'(x)-T(x)|>C\varepsilon''\lambda\right\}\leq \exp(-\lambda){\rm vol}(B)
\eeq
{\rm and slow twists in Example~\ref{e:Example1} show that the estimate above is sharp. The set 
\beq
\left\{x\in B:|\Phi'(x)-T(x)|> C\varepsilon''\lambda\right\}
\eeq may well be small in a more refined sense but we did not pursue this investigation in \cite{DF4}}.

\subsection{Approximation by Euclidean Motions; Proper and Improper.}
From Definition~\ref{d:block}, we recall that $V_l(z_0,...,z_l)$ denotes the $l$-dimensional volume of the $l$-simplex with vertices at $z_0,...,z_l$. We now have:

\begin{thm}
\item[(a)] Let $\varepsilon>0$. Let $\Phi:\mathbb R^D\to \mathbb R^D$ be a $\varepsilon$-distorted diffeomorphism. Let $z\in \mathbb R^D$ and $r>0$ be given.  Then, there exists an Euclidean motion $A=A_B$ with  
$B=B(z,r)$,  such that for $x\in B(z,r)$,
\begin{itemize}
\item[(1)] $|\Phi(x)-A(x)|\leq C\varepsilon r$.
\item[(2)] Moreover, $A$ is proper iff $\Phi$ is proper.
\end{itemize}
\item[(b)] Let $x_0,...,x_D\in \mathbb R^D$ with ${\rm diam}\left\{x_0,...,x_D\right\}\leq 1$ and $V_D(x_0,...,x_D)\geq \eta^D$ where
$0<\eta<1$ and let $0<\beta<c'\eta^D$ for a small enough $c'$. Let $\Phi:\mathbb R^D\to\mathbb R^D$ be a $\beta$-distorted diffeomorphism. Finally let $T$ be the one and only one affine map that agrees with
$\Phi$ on $\left\{x_0,...,x_D\right\}$. (We recall that the existence and uniqueness of such $T$ follows from \cite{ATV}, T may not be invertible). Then $\Phi$ is proper iff $T$ is proper.
\label{t:theorememotionapprox}
\end{thm}

Theorem~\ref{t:lemmatwists} then follows from (\ref{e:lemmaone}), (\ref{e:lemmatwo}) and (\ref{e:corollaryone}). $\Box$

We are now going to introduce the notation $E$ for a special finite set in $\mathbb R^D$ whose diameter satisfies ${\rm diam}(E)\leq 1$ and its points are well separated (we will define this more precisely in a moment and throughout the paper).
The aim of the next section is to show that for certain such sets $E$, 
we can always construct an improper $\varepsilon$-distorted
diffeomorphism $\Phi:\mathbb R^D\to\mathbb R^D$ such that $\Phi(z)=z$ for each $z\in E$, $\Phi$ agrees with a improper Euclidean motion $A_z:\mathbb R^D\to \mathbb R^D$ in a ball of small enough radius (in particular smaller than the maximum separation distance between points in $E$) and with center $z$ for each $z\in E$ and $\Phi$ agrees with a improper Euclidean motion on all points in $\mathbb R^D$ whose distance to $E$ is large enough. Such sets $E$ we will use to define our special constant $K$ in (~\ref{e:specialk}). 

Our technique will be one of Approximate Reflections. Thus we intoduce our next section of:

\section{Approximate Reflections from $\mathbb R^D$ to $\mathbb R^D$.}
\setcounter{equation}{0}
Suppose that $S$ is a finite subset of a affine hyperplane $H\subset \mathbb R^D$. (So $H$ has dimension $D-1$). Let $A:\mathbb R^D\to \mathbb R^D$ denote reflection through $H$. Then 
$A$ is an improper Euclidean motion and $A(z)=z$ for each $z\in S$. \footnote{For easy understanding: Suppose $D=2$ and $H$ is a line with the set $S$ on the line. Let $A$ denote reflection of the lower half plane to the upper half plane through $S$. Then $A$ is a Euclidean motion and fixes points on $S$ because it is an isometry.} Now suppose that $S$ is again a finite subset of a affine hyperplane $H\subset \mathbb R^D$ and assume that 
we have a map $p:\mathbb R^D \to \mathbb R^D$ with $p(z)$ close to $z$  on $S$. (We will define this more precisely in a moment).  Then we call $p$ an approximate reflection through $H$.
We will now use approximate reflections to construct $\varepsilon$-distorted diffeomorphisms as we have described above.

We proceed and for our set $S$, we will now take a special set $E\subset \mathbb R^D$ as mentioned above which we now define precisely in our main result of this next subsection:
\subsection{Theorem~\ref{t:lemmareflection3}.}
We have:

\begin{thm}
Let $\varepsilon>0$, $0<\tau<1$, $E\subset \mathbb R^D$ be a finite set with ${\rm diam}(E)=1$ and $|z-z'|\geq \tau$ for all $z,z'\in E$ distinct. Assume that $V_D(E)\leq \eta^{D}$ where $0<\eta<c\tau\varepsilon$ for small enough $c$. Here we recall $V_D$ is given by Definition~\ref{d:block}. Then, there exists a $C\varepsilon$-distorted diffeomorphism $\Phi:\mathbb R^D\to \mathbb R^D$ with the following properties:
\begin{itemize}
\item[(a)] $\Phi$ coincides with an improper Euclidean motion on $\left\{x\in \mathbb R^D:\, {\rm dist}(x, E)\geq 20\right\}$.
\item[(b)] $\Phi$ coincides with an improper Euclidean motion $A_z$ on $B(z,\tau/100)$ for each $z\in E$.
\item[(c)] $\Phi(z)=z$ for each $z\in E$.
\end{itemize}
\label{t:lemmareflection3}
\end{thm}

\section{The constant $K'$ and Theorem~\ref{t:lemmareflection3} for proper $\varepsilon$-distorted diffeomorphisms.}
\setcounter{equation}{0}
We are now able to prove an analogy of Theorem~\ref{t:lemmareflection3} for proper $\varepsilon$-distorted diffeomorphisms, namely Theorem~\ref{t:lemmareflectionex4}.
Two important ingredients we will need to do this will be to assume a separation condition on the points of $E$ as in Theorem~\ref{t:lemmareflection3} and also a condition on the cardinality of $E$ which will be our constant $K$, see (\ref{e:specialk}) below.

Thus we have as our main result in this section:

\begin{thm}
Let $\phi:E\to \mathbb R^D$ with $E\subset \mathbb R^D$ finite. Let $\varepsilon>0$ and $0<\tau<1$. 
We make the following assumptions:
\begin{itemize}
\item Assumption on parameters:
\begin{itemize}
\item Let $0<\eta<c\varepsilon \tau$ for small enough $c$.
\item Let $C_K\beta^{1/\rho_k}\tau^{-1}<{\rm min}(\varepsilon,\eta^D)$ for some large enough $C_K>0$ and $\rho_K>0$, the later also depending only on 
$D$ and $K$.
\end{itemize}
\item Assumptions on E: ${\rm diam}(E)=1$,\, $|x-y|\geq \tau$, for any $x,y\in E$ distinct, 
\beq {\rm card}(E)\leq K.
\label{e:specialk}
\eeq
\item Assumption on $\phi$: $\phi$ has no negative $\eta$-blocks and
\[
(1+\beta)^{-1}|x-y|\leq |\phi(x)-\phi(y)|\leq (1+\beta)|x-y|,\, x, y\in E.
\]
\end{itemize}
Then, there exists a proper $C\varepsilon$-distorted diffeomorphism $\Phi:\mathbb R^D\to \mathbb R^D$ with the following
properties:
\begin{itemize}
\item $\Phi=\phi$ on $E$.
\item $\Phi$ agrees with an Euclidean motion $A_{\infty}$ on
$\left\{x\in \mathbb R^D:\, {\rm dist}(x,E)\geq 1000\right\}.$
\item For each $z\in E$, $\Phi$ agrees with a Euclidean motion $A_z$ on $B(z,\tau/1000)$.
\label{t:lemmareflectionex4}
\end{itemize}
\end{thm}

The rest of this paper proves Theorem~\ref{t:0.1}.

\section{\textsc{Preliminary Results}} \label{s:1}
\setcounter{thm}{0}
\setcounter{equation}{0}
Throughout this paper, we adopt the following conventions regarding constants. A ``controlled constant'' is a constant determined by $c_{0}$, $C_{1}$, $c_{2}$, $n$ in Section 2. We write $c, \, C, \, C'$, etc. to denote controlled constants. When we write e.g., $c_{0}, \, C_{17}$, or $c_{5}$, then this expression denotes the same controlled constant in every occurrence. When we write e.g. $c, \, C$, or $C'$, without a numerical subscript, then this expression may denote different controlled constants in different occurrences. We write $B(z,r)$ to denote the closed ball in $\bbR^{n}$ with center $z$, radius $r$.

\begin{lem}
	\label{l:1.1}
	\addcontentsline{toc}{subsection}{Lemma 1.1}
	Let $B(z,r) \subset E$. Then there exists a Euclidean motion $A$, such that for every $K \geq 1$, and for every
	\beq
		\label{eq:1.1}
		y \in B(z, Kr) \cap E, \text{ we have } \ml \phi(y) - A(y) \mr \leq CK^{2}\ep r.
	\eeq
\end{lem}

\begin{proof}
	Without loss of generality, we may assume $z=0$, $r=1$, and $\phi(0)=0$. Let $e_{1},\cdots,e_{n}$ be the unit vector in $\bbR^{n}$. By \eqref{i:0.1} we have $1-\ep \leq \ml \phi(e_{i}) \mr \leq 1+\ep$ for each $i$, and $(1-\ep)\sqrt{2} \leq \ml \phi(e_{i}) - \phi(e_{j}) \mr \leq (1+\ep)\sqrt{2}$ for $i \neq j$. Since $-2\phi(e_{i})\cdot\phi(e_{j}) = \ml \phi(e_{i}) - \phi(e_{j}) \mr^{2} - \ml \phi(e_{i}) \mr^{2} - \ml \phi(e_{j}) \mr^{2}$, it follows that
	\beq
		\label{eq:1.2}
		\ml \phi(e_{i}) \cdot \phi(e_{j}) - \delta_{ij} \mr \leq C\ep \text{ for each $i,j$, where $\delta_{ij}$ is the Kronecker delta function}.
	\eeq
	Let $A \in \OO(n)$ be the orthogonal matrix whose columns arise by applying the Gram-Schmidt process to the vectors $\phi(e_{1}),\,\phi(e_{2}),\,\cdots,\,\phi(e_{n})$. Then \eqref{eq:1.2} implies the estimate
	\[
		\ml \phi(e_{i}) - Ae_{i} \mr \leq C\ep \text{ for each $i$}.
	\]
	Replacing $\phi$ by $A^{-1}\circ\phi$, we may therefore assume without loss of generality that
	\beq
		\label{eq:1.3}
		\ml \phi(e_{i}) - e_{i} \mr \leq C\ep \text{ for each $i$}.
	\eeq
	Assume \eqref{eq:1.3}, and recalling that $\phi(0) = 0$, we will prove \eqref{eq:1.1} with $A = I$. Thus, let $K \geq 1$, and let $y \in B(0,K) \cap E$. By \eqref{i:0.1}, we have $(1-\ep)\ml y \mr \leq \ml \phi(y) \mr \leq (1+\ep)\ml y \mr$, hence
	\beq
		\label{eq:1.4}
		\Big| | \phi(y) | - | y | \Big| \leq \ep K.
	\eeq
	In particular,
	\beq
		\label{eq:1.5}
		\ml \phi(y) \mr \leq zK.
	\eeq
	Again applying \eqref{i:0.1}, we have
	\[
		(1-\ep)\ml y-e_{i} \mr \leq \ml \phi(y)-\phi(e_{i}) \mr \leq (1+\ep) \ml y-e_{i} \mr \text{ for each $i$}.
	\]
	Hence, by \eqref{eq:1.3} and \eqref{eq:1.5}, we have
	\beq
		\label{eq:1.6}
		\Big| |\phi(y)-e_{i}| - |y-e_{i}| \Big| \leq C\ep K \text{ for each $i$}.
	\eeq
	From \eqref{eq:1.4}, \eqref{eq:1.5}, \eqref{eq:1.6}, we see that
	\beq
		\label{eq:1.7}
		\Big| |\phi(y)|^{2} - |y|^{2} \Big| = \Big( |\phi(y)| + |y| \Big) \cdot \Big| |\phi(y)| - |y| \Big| \leq CK^{2}\ep,
	\eeq
	and similarly,
	\beq
		\label{eq:1.8}
		\Big| |\phi(y)-e_{i}|^{2} - |y-e_{i}|^{2} \Big| \leq CK^{2}\ep.
	\eeq
	Since
	\begin{align*}
		-2\phi(y) \cdot e_{i} &= |\phi(y)-e_{i}|^{2} - |\phi(y)|^{2} - 1 \text{ and} \\
		-2y \cdot e_{i} &= |y-e_{i}|^{2} - |y|^{2} - 1,
	\end{align*}
	it follows from \eqref{eq:1.7} and \eqref{eq:1.8} that
	\[
		\Big| \big[ \phi(y) - y \big] \cdot e_{i} \Big| \leq CK^{2}\ep \text{ for each $i$}.
	\]
	Consequently, $|\phi(y) - y| \leq CK^{2}e$, proving \eqref{eq:1.1} and $A =$ identity.
\end{proof}

When we apply Lemma \ref{l:1.1}, we will always take $K$ to be a controlled constant.

The proof of the following Lemma is straightforward, and may be left to the reader. Note that $\nabla A(x)$ is independent of $x \in \bbR^{n}$ when $A: \bbR^{n} \to \bbR^{n}$ is an affine map. We write $\nabla A$ in this case without indicating $x$.

\begin{lem}
	\label{l:1.2}
	\addcontentsline{toc}{subsection}{Lemma 1.2}
	Let $B(z,r)$ be a ball, let $A: \bbR^{n} \to \bbR^{n}$ be an affine map, and let $M > 0$ be a real number. If $|A(y)| \leq M$ for all $y \in B(z,r)$, then $|\nabla A| \leq CM/r$, and for any $K \geq 1$ and $y \in B(z,Kr)$ we have
	\[
		\big| A(y) \big| \leq CKM.
	\]
\end{lem}

When we apply Lemma \ref{l:1.2}, we will always take $K$ to be a controlled constant.

\begin{lem}
	\label{l:1.3}
	\addcontentsline{toc}{subsection}{Lemma 1.3}
	For $\eta > 0$ small enough, we have
	\beq
		\label{eq:1.9}
		\big(1-C\ep\big)I \leq \Big(\nabla\phi(y)\Big)^{+}\Big(\nabla\phi(y)\Big) \leq (1+C\ep)I \text{ for all $y\in\bbR^{n}$ s.t. $d(y) < \eta$}.
	\eeq
\end{lem}

\begin{proof}
	If $y$ is an interior point of $E$, then \eqref{eq:1.9} follows easily from \eqref{i:0.1}. Suppose $y$ is a boundary point of $E$. Arbitrarily close to $y$, we can find $x \in \bbR^{n} \backslash E$. Applying \eqref{i:0.2}, we obtain an interior point $z$ in $E$ such that $|z-x| \leq C_{1}d(x) \leq C_{1}|y-x|$, hence $|z-y| \leq (1+C_{1})|y-x|$. Since $z$ is an interior point of $E$, we have
	\beq
		\label{eq:1.10}
		\big(1-C\ep\big)I \leq \Big( \nabla\phi(z) \Big)^{+}\Big( \nabla\phi(z) \Big) \leq (1+C\ep),
	\eeq
	as observed above. However, we can make $|z-y|$ as small as we like here, simply by taking $|y-x|$ small enough. Since $\phi \in C^{1}(U)$, we may pass to the limit, and deduce \eqref{eq:1.9} from \eqref{eq:1.10}. Thus, \eqref{eq:1.9} holds for all $y \in E$. 	
	Since $E \subset U$ is compact and $\phi \in C^{1}(U)$, the lemma now follows.
\end{proof}

\begin{lem}
	\label{l:1.4}
	\addcontentsline{toc}{subsection}{Lemma 1.4}
	For $\eta > 0$ small enough, we have
	\[
		\Big| \phi(y) - \big[ \phi(x) + \nabla\phi(x)\cdot(y-x) \big] \Big| \leq \ep \big|y-x\big|
	\]
	for all $x,y \in U$ such that $d(x) \leq \eta$ and $|y-x| \leq \eta$.
\end{lem}

\begin{proof}
	If $\eta$ is small enough and $d(x) \leq \eta$, then $B(x,\eta) \subset U$ and $|\nabla\phi(y) - \nabla\phi(x)| \leq \ep$ for all $y \in B(x,\eta)$. (These remarks follow from the fact that $E \subset U$ is compact and $\phi \in C^{1}(U)$.)
	
	The lemma now follows from the fundamental theorem of calculus.
\end{proof}

\begin{lem}
	\label{l:1.5}
	\addcontentsline{toc}{subsection}{Lemma 1.5}
	Let $\Psi: \bbR^{n} \to \bbR^{n}$ be a $C^{1}$ map. Assume that $\det\nabla\Psi\neq 0$ everywhere on $\bbR^{n}$, and assume that $\Psi$ agrees with a Euclidean motion outside a ball B. Then $\Psi: \bbR^{n} \to \bbR^{n}$ is one-to-one and onto.
\end{lem}

\begin{proof}
	Without loss of generality, we may suppose $\Psi(x) = x$ for $|x| \geq 1$. First we show that $\Psi$ is onto. Since $\det\nabla\Psi \neq 0$, we know that $\Psi(\bbR^{n})$ is open, and of course $\Psi(\bbR^{n})$ is non-empty. If we can show that $\Psi(\bbR^{n})$ is closed, then it follows that $\Psi(\bbR^{n}) = \bbR^{n}$, i.e., $\Psi$ is onto.
	
	Let $\{x_{\nu}\}_{\nu \geq 1}$ be a sequence converging to $x_{\infty} \in \bbR^{n}$, with each $x_{\nu} \in \Psi(\bbR^{n})$. We show that $x_{\infty} \in \Psi(\bbR^{n})$. Let $x_{\nu} = \Psi(y_{\nu})$. If infinitely many $y_{\nu}$ satisfy $|y_{\nu}| \geq 1$, then infinitely many $x_{\nu}$ satisfy $|x_{\nu}| \geq 1$, since $x_{\nu} = \Psi(y_{\nu}) = y_{\nu}$ for $|y_{\nu}| \geq 1$. Hence, $|x_{\infty}| \geq 1$ in this case, and consequently
	\[
		x_{\infty} = \Psi(x_{\infty}) \in \Psi(\bbR^{n}).
	\]
	
	On the other hand, if only finitely many $y_{\nu}$ satisfy $|y_{\nu}| \geq 1$, then there exists a convergent subsequence $y_{\nu_{i}} \to y_{\infty}$ as $i \to \infty$. In this case, we have
	\[
		x_{\infty} = \lim_{i \to \infty}\Psi(y_{\nu_{i}}) = \Psi(y_{\infty}) \in \Psi(\bbR^{n}).
	\]
	
	Thus, in all cases, $x_{\infty} \in \Psi(\bbR^{n})$. This proves that $\Psi(\bbR^{n})$ is closed, and therefore $\Psi: \bbR^{n} \to \bbR^{n}$ is onto.
	
	Let us show that $\Psi$ is one-to-one. We know that $\Psi$ is bounded on the unit ball. Fix $M$ such that
	\beq
		\label{eq:1.11}
		\big|\Psi(y)\big| \leq M \ \text{for} \ |y| \leq 1.
	\eeq
	We are assuming that $\Psi(y) = y$ for $|y| \geq 1$. For $|x| > \max(M,1)$, it follows that $y=x$ is the only point $y \in \bbR^{n}$ such that $\Psi(y) = x$. Now let $Y = \{ y'\in\bbR^{n}: \Psi(y') = \Psi(y'') \text{ for some } y'' \neq y' \}$. The set $Y$ is bounded, thanks to \eqref{eq:1.11}. Also, the inverse function theorem shows that $Y$ is open. We will show that $Y$ is closed. This implies that $Y$ is empty, proving that $\Psi: \bbR^{n} \to \bbR^{n}$ is one-to-one.
	
	Thus, let $\{y_{\nu}'\}_{\nu \geq 1}$ be a convergent sequence, with each $y_{\nu}' \in Y$; suppose $y_{\nu}' \to y_{\infty}'$ as $\nu \to \infty$. We will prove that $y_{\infty}' \in Y$.
	
	For each $\nu$, pick $y_{\nu}'' \neq y_{\nu}'$ such that
	\[
		\Psi(y_{\nu}'') = \Psi(y_{\nu}').
	\]
	Each $y_{\nu}''$ satisfies $|y_{\nu}''| \leq \max(M.1)$, thanks to \eqref{eq:1.11}.
	
	Hence, after passing to a subsequence, we may assume $y_{\nu}'' \to y_{\infty}''$ as $\nu \to \infty$. We already know that $y_{\nu}' \to y_{\infty}'$ as $\nu \to \infty$.
	
	Suppose $y_{\infty}' = y_{\infty}''$. Then arbitrarily near $y_{\infty}'$ there exist pairs $y_{\nu}',\,y_{\nu}''$, with $y_{\nu}' \neq y_{\nu}''$ and $\Psi(y_{\nu}') = \Psi(y_{\nu}'')$. This contradicts the inverse function theorem, since $\det\nabla\Psi(y_{\infty}') \neq 0$.
	
	Consequently, we must have $y'_{\infty} \neq y''_{\infty}$. Recalling that $\Psi(y'_{\nu}) = \Psi(y''_{\nu})$, and passing to the limit, we see that $\Psi(y'_{\infty}) = \Psi(y''_{\infty})$.
	
	By definition, we therefore have $y'_{\infty} \in Y$, proving that $Y$ is closed, as asserted above. Hence, $Y$ is empty, and $\Psi: \bbR^{n} \to \bbR^{n}$ is one-to-one.
\end{proof}

From now on, we assume without loss of generality that
\beq
	\label{eq:1.12}
	\dia E = 1.
\eeq

\section{\textsc{Whitney Technology}} \label{s:2}
\setcounter{thm}{0}
\setcounter{equation}{0}
From the standard proof [\cite{F3}] of the Whitney extension theorem, we obtain the following,

\ul{Whitney cubes}: $\bbR^{n} \backslash E$ is partitioned into ``Whitney cubes'' $\{Q_{\nu}\}$. We write $\beta_{\nu}$ to denote the sidelength of $Q_{\nu}$, and we write $Q_{\nu}^{*}$ to denote the cube $Q_{\nu}$, dilated about its center by a factor of 3. The Whitney cubes have the following properties,
\beq
	\label{eq:2.1} \text{$c\beta_{\nu} \leq d(x) \leq C\beta_{\nu}$ for all $x \in Q_{\nu}^{*}$}.
\eeq
\beq
	\label{eq:2.2} \text{Any given $x \in \bbR^{n}$ belongs to $Q_{\nu}^{*}$ for at most $C$ distinct $\nu$}.
\eeq

\ul{Whitney partition of unity}: For each $Q_{\nu}$, we have a cutoff function $\Theta_{\nu} \in C^{\infty}(\bbR^{n})$, with the following properties,
\beq
	\label{eq:2.3} \text{$\Theta_{\nu} \geq 0$ on $\bbR^{n}$}.
\eeq
\beq
	\label{eq:2.4} \text{supp}\,\Theta_{\nu} \subset Q_{\nu}^{*}.
\eeq
\beq
	\label{eq:2.5} \text{$|\nabla\Theta_{\nu}| \leq C\beta_{\nu}^{-1}$ on $\bbR^{n}$}.
\eeq
\beq
	\label{eq:2.6} \text{$\sum_{\nu}\Theta_{\nu} = 1$ on $\bbR^{n} \backslash E$}.
\eeq

\ul{Regularized Distance}: A function $\delta(x)$, defined on $\bbR^{n}$, has the following properties,
\beq
	\label{eq:2.7} \text{$cd(x) \leq \delta(x) \leq Cd(x)$ for all $x \in \bbR^{n}$}.
\eeq
\beq
	\label{eq:2.8} \text{$\delta(\cdot)$ belongs to $C_{\text{loc}}^{\infty}(\bbR^{n} \backslash E)$}.
\eeq
\beq
	\label{eq:2.9} \text{$|\nabla\delta(x)| \leq C$ for all $x \in \bbR^{n} \backslash E$}.
\eeq

Thanks to \eqref{eq:2.1} and \eqref{eq:2.7}, the following holds,
\beq
	\label{eq:2.10}
	\left[\begin{array}{l}
		\text{Let $x \in \bbR^{n}$, and let $Q_{\nu}$ be one of the Whitney cubes}. \\
		\text{If $d(x) \geq c_{0}$ and $x \in Q_{\nu}^{*}$, then $\beta_{\nu} > c_{3}$}.
	\end{array}\right.
\eeq

We bring \ref{i:0.2} into the picture. Recall that $\dia E = 1$.

Let $Q_{\nu}$ be a Whitney cube such that $\beta_{\nu} \leq c_{3}$. Then $d(x) < c_{0}$ for all $x \in Q_{\nu}^{*}$, as we see from \eqref{eq:2.10}. Let $x_{\nu}$ be the center of $Q_{\nu}$. Since $d(x_{\nu}) < c_{0}$, we may apply \ref{i:0.2} with $x = x_{\nu}$.

Thus, we obtain a ball
\beq
	\label{eq:2.11} B(z_{\nu},r_{\nu}) \subset E,
\eeq
such that
\beq
	\label{eq:2.12} cd(x_{\nu}) < r_{\nu} \leq Cd(x_{\nu}),
\eeq
and
\beq
	\label{eq:2.13} |z_{\nu} - x_{\nu}| \leq Cd(x_{\nu}).
\eeq

The ball $B(z_{\nu},r_{\nu})$ has been defined whenever $\beta_{\nu} \leq c_{3}$. \big(To see that $r_{\nu} \leq Cd(x_{\nu})$, we just note that $B(z_{\nu},r_{\nu}) \subset E$ but $x_{\nu} \notin E$; hence $|z_{\nu}-x_{\nu}| > r_{\nu}$, and therefore \eqref{eq:2.13} implies $r_{\nu} \leq Cd(x_{\nu})$.\big)

From \eqref{eq:2.12}, \eqref{eq:2.13} and \eqref{eq:2.1}, \eqref{eq:2.7}, we learn the following,
\beq
	\label{eq:2.14} Q_{\nu}^{*} \subset B(z_{\nu},Cr_{\nu}).
\eeq
\beq
	\label{eq:2.15} \text{$c\delta(x) < r_{\nu} < C\delta(x)$ for any $x \in Q_{\nu}^{*}$}.
\eeq
\beq
	\label{eq:2.16} \text{$|z_{\nu} - x| \leq C\delta(x)$ for any $x \in Q_{\nu}^{*}$}.
\eeq

These results \big(and \eqref{eq:2.11}\big) in turn imply the following,
\beq
	\label{eq:2.17} \text{Let $x \in Q_{\mu}^{*} \cap Q_{\nu}^{*}$. Then $B(z_{\nu},r_{\nu}) \subset B(z_{\mu},Cr_{\mu}) \cap E$}.
\eeq
Here, \eqref{eq:2.14}, \eqref{eq:2.15}, \eqref{eq:2.16} hold whenever $\beta_{\nu} \leq c_{3}$; while \eqref{eq:2.17} holds whenever $\beta_{\mu},\beta_{\nu} \leq c_{3}$.

We want an analogue of $B(z_{\nu},r_{\nu})$ for Whitney cubes $Q_{\nu}$ such that $\beta_{\nu} > c_{3}$.

There exists $x \in \bbR^{n}$ such that $d(x) = c_{0}/2$. Applying \ref{i:0.2} to this $x$, we obtain a ball
\beq
	\label{eq:2.18} B(z_{\infty}, r_{\infty}) \subset E,
\eeq
such that
\beq
	\label{eq:2.19} c < r_{\infty} \leq 1/2.
\eeq
(We have $r_{\infty} \leq 1/2$, simply because $\dia E = 1$.)

From \eqref{eq:2.18}, \eqref{eq:2.19} and the fact that $\dia E = 1$, we conclude that
\beq
	\label{eq:2.20} E \subset B(z_{\infty}, Cr_{\infty}).
\eeq

\section{\textsc{Picking Euclidean Motions}} \label{s:3}
\setcounter{thm}{0}
\setcounter{equation}{0}
For each Whitney cube $Q_{\nu}$, we pick a Euclidean motion $A_{\nu}$, as follows,
\vspace{-1em}
\begin{description}[leftmargin=0em]
	\item[Case I] (``Small'' $Q_{\nu}$). Suppose $\beta_{\nu} \leq c_{3}$. Applying Lemma \ref{l:1.1} to the ball $B(z_{\nu},r_{\nu})$, we obtain a Euclidean motion $A_{\nu}$ with the following property.
	\beq
		\label{eq:3.1}
		\text{For $K \geq 1$ and $y \in B(z_{\nu},Kr_{\nu}) \cap E$, we have $\big| \phi(y) - A_{\nu}(y) \big| \leq CK^{2}\ep r_{\nu}$}.
	\eeq
	\item[Case II] (``Not-so-small'' $Q_{\nu}$). Suppose $\beta_{\nu} > c_{3}$. Applying Lemma \ref{l:1.1} to the ball $B(z_{\infty},r_{\infty})$, we obtain a Euclidean motion $A_{\infty}$ with the following property.
	\beq
		\label{eq:3.2}
		\text{For $K \geq 1$ and $y \in B(z_{\infty},Kr_{\infty}) \cap E$, we have $\big| \phi(y) - A_{\infty}(y) \big| \leq CK^{2}\ep r_{\infty}$}.
	\eeq
	In case II, we define
	\beq
		\label{eq:3.3}
		A_{\nu} = A_{\infty}.
	\eeq
	Thus, $A_{\nu} = A_{\nu'}$ whenever $\nu$ and $\nu'$ both fall into Case II. Note that \eqref{eq:3.2} together with \eqref{eq:2.19} and 
\eqref{eq:2.20} yield the estimate
	\beq
		\label{eq:3.4}
		\big| \phi(y) - A_{\infty}(y) \big| \leq C\ep \text{ for all } y \in E.
	\eeq
\end{description}

The next result establishes the mutual consistency of the $A_{\nu}$.

\begin{lem}
	\label{l:3.1}
	\addcontentsline{toc}{subsection}{Lemma 3.1}
	For $x \in Q_{\mu}^{*} \cap Q_{\nu}^{*}$, we have
	\beq
		\label{eq:3.5}
		\big| A_{\mu}(x) - A_{\nu}(x) \big| \leq C\ep\delta(x),
	\eeq
	and
	\beq
		\label{eq:3.6}
		\big| \nabla A_{\mu} - \nabla A_{\nu} \big| \leq C\ep.
	\eeq
\end{lem}

\begin{proof}
	We proceed by cases.
	\vspace{-1em}
	\begin{description}[leftmargin=0em]
		\item[Case 1:] Suppose $\beta_{\mu},\,\beta_{\nu} \leq c_{3}$. Then $A_{\nu}$ satisfies \eqref{eq:3.1}, and $A_{\mu}$ satisfies the analogous condition for $B(z_{\mu},r_{\mu})$. Recalling \eqref{eq:2.17}, we conclude that
		\beq
			\label{eq:3.7}
			\big| \phi(y) - A_{\mu}(y) \big| \leq C\ep r_{\mu} \text{ for } y \in B(z_{\nu},r_{\nu}),
		\eeq
		and
		\beq
			\label{eq:3.8}
			\big| \phi(y) - A_{\nu}(y) \big| \leq C\ep r_{\nu} \text{ for } y \in B(z_{\nu},r_{\nu}).
		\eeq
		Moreover, \eqref{eq:2.15} gives
		\beq
			\label{eq:3.9}
			c\delta(x) < r_{\mu} < C\delta(x) \text{ and } c\delta(x) < r_{\nu} < C\delta(x).
		\eeq
		By \eqref{eq:3.7}, \eqref{eq:3.8}, \eqref{eq:3.9}, we have
		\beq
			\label{eq:3.10}
			\big| A_{\mu}(y) - A_{\nu}(y) \big| \leq C\ep r_{\nu} \text{ for } y \in B(z_{\nu},r_{\nu}).
		\eeq
		Now, $A_{\mu}(y) - A_{\nu}(y)$ is an affine function. Hence, Lemma \ref{l:1.2} and inclusion \eqref{eq:2.14} allow us to deduce from \eqref{eq:3.10} that
		\beq
			\label{eq:3.11}
			\big|  A_{\mu}(y) - A_{\nu}(y) \big| \leq C\ep r_{\nu} \text{ for all } y \in Q_{\nu}^{*},
		\eeq
		and
		\beq
			\label{eq:3.12}
			\big| \nabla A_{\mu} - \nabla A_{\nu} \big| \leq C\ep.
		\eeq
		Since $x \in Q_{\nu}^{*}$, the desired estimates \eqref{eq:3.5}, \eqref{eq:3.6} follow at once from \eqref{eq:3.9}, \eqref{eq:3.11} and \eqref{eq:3.12}. Thus, Lemma \ref{l:3.1} holds in Case 1.
		
		\item[Case 2:] Suppose $\beta_{\nu} \leq c_{3}$ and $\beta_{\mu} > c_{3}$. Then by \eqref{eq:3.1} and \eqref{eq:2.11}, $A_{\nu}$ satisfies
		\beq
			\label{eq:3.13}
			\big| \phi(y) - A_{\nu}(y) \big| \leq C\ep r_{\nu} \text{ for }  y \in B(z_{\nu},r_{\nu});
		\eeq
		whereas $A_{\mu} = A_{\infty}$, so that \eqref{eq:3.4} and \eqref{eq:2.11} give
		\beq
			\label{eq:3.14}
			\big| \phi(y) - A_{\mu}(y) \big| \leq C\ep \text{ for all } y \in B(z_{\nu},r_{\nu}).
		\eeq
		Since $x \in Q_{\mu}^{*} \cap Q_{\nu}^{*}$, \eqref{eq:2.1} and \eqref{eq:2.7} give
		\[
			c\delta(x) \leq \beta_{\mu} \leq C\delta(x) \text{ and } c\delta(x) \leq \beta_{\nu} \leq C\delta(x).
		\]
		In this case, we have also $\beta_{\nu} \leq c_{3}$ and $\beta_{\mu} > c_{3}$. Consequently,
		\beq
			\label{eq:3.15}
			c < \beta_{\mu} < C, c < \beta_{\nu} < C, \text{ and } c < \delta(x) < C.
		\eeq
		By \eqref{eq:2.15}, we have also
		\beq
			\label{eq:3.16}
			c < r_{\nu} < C.
		\eeq
		
		From \eqref{eq:3.13}, \eqref{eq:3.14}, \eqref{eq:3.16}, we see that
		\beq
			\label{eq:3.17}
			\big|  A_{\mu}(y) - A_{\nu}(y) \big| \leq C\ep \text{ for all } y \in B_{\nu}(z_{\nu},r_{\nu}).
		\eeq
		
		Lemma \ref{l:1.2}, estimate \eqref{eq:3.16} and inclusion \eqref{eq:3.14} let us deduce from \eqref{eq:3.17} that
		\beq
			\label{eq:3.18}
			\big|  A_{\mu}(y) - A_{\nu}(y) \big| \leq C\ep \text{ for all } y \in Q_{\nu}^{*},
		\eeq
		and
		\beq
			\label{eq:3.19}
			\big| \nabla A_{\mu} - \nabla A_{\nu} \big| \leq C\ep.
		\eeq
		Since $x \in Q_{\nu}^{*}$, the desired estimates \eqref{eq:3.5}, \eqref{eq:3.6} follow at once from \eqref{eq:3.15}, \eqref{eq:3.18}, \eqref{eq:3.19}. Thus, Lemma \ref{l:3.1} holds in Case 2.
		
		\item[Case 3:] Suppose $\beta_{\nu} > c_{3}$ and $\beta{\mu} \leq c_{3}$. Reversing the roles of $Q_{\mu}$ and $Q_{\nu}$, we reduce matters to Case 2. Thus, Lemma \ref{l:3.1} holds in Case 3.
		
		\item[Case 4:] Suppose $\beta_{\mu},\,\beta_{\nu} > c_{3}$. Then by definition $A_{\mu} = A_{\nu} = A_{\infty}$, and estimates \eqref{eq:3.5}, \eqref{eq:3.6} hold trivially. Thus, Lemma \ref{l:3.1} holds in Case 4.
	\end{description}
	
	We have proved the desired estimates \eqref{eq:3.5}, \eqref{eq:3.6} in all cases.
\end{proof}

The following lemma shows that $A_{\nu}$ closely approximates $\phi$ on $Q_{\nu}^{*}$ when $Q_{\nu}^{*}$ lies very close to $E$.

\begin{lem}
	\label{l:3.2}
	\addcontentsline{toc}{subsection}{Lemma 3.2}
	For $\eta > 0$ small enough, the following holds. \\
	Let $x \in Q_{\nu}^{*}$, and suppose $\delta(x) \leq \eta$. Then $x \in U$, $|\phi(x) - A_{\nu}(x)| \leq C\ep\delta(x)$, and $|\nabla\phi(x) - \nabla A_{\nu}| \leq C\ep$.
\end{lem}

\begin{proof}
	We have $\beta_{\nu} < C\delta(x) \leq C\eta$ by \eqref{eq:2.1} and \eqref{eq:2.7}. If $\eta$ is small enough, it follows that $\beta_{\nu} < c_{3}$, so $Q_{\nu}$ falls into Case I, and we have
	\beq
		\label{eq:3.20}
		\big| \phi(y) - A_{\nu}(y) \big| \leq C\ep r_{\nu} \text{ for } y \in B(z_{\nu},r_{\nu})
	\eeq
	by \eqref{eq:3.1}. Also, \eqref{eq:2.15}, \eqref{eq:2.16} show that
	\beq
		\label{eq:3.21}
		B(z_{\nu},r_{\nu}) \subset B\big(x, C\delta(x)\big) \subset B(x,C\eta).
	\eeq
	We have
	\beq
		\label{eq:3.22}
		d(x) \leq C\delta(x) \leq C\eta
	\eeq
	by \eqref{eq:2.7}. (In particular, $x \in U$ if $\eta$ is small enough.) If $\eta$ is small enough, then \eqref{eq:3.21}, \eqref{eq:3.22} and Lemma \ref{l:1.4} imply
	\[
		y \in U \text{ and } \Big| \phi(y) - \big[ \phi(x) + \nabla\phi(x) \cdot (y-x) \big] \Big| < \ep \big| y-x \big| \text{ for } y \in B(z_{\nu},r).
	\]
	Hence, by \eqref{eq:3.21} and \eqref{eq:2.15}, we obtain the estimate
	\beq
		\label{eq:3.23}
		\Big| \phi(y) - \big[ \phi(x) + \nabla\phi(x) \cdot (y-x) \big] \Big| \leq C\ep r_{\nu} \text{ for } y \in B(z_{\nu},r_{\nu}).
	\eeq
	Combing \eqref{eq:3.20} with \eqref{eq:3.23}, we find that
	\beq
		\label{eq:3.24}
		\Big| A_{\nu}(y) - \big[ \phi(x) + \nabla\phi(x) \cdot (y-x) \big] \Big| \leq C\ep r_{\nu} \text{ for } y \in B(z,r_{\nu}).
	\eeq
	The function $y \to A_{\nu}(y) -  [ \phi(x) + \nabla\phi(x) \cdot (y-x) ]$ is affine. Hence, estimate \eqref{eq:3.24}, inclusion \eqref{eq:2.14}, and Lemma \ref{l:1.2} together tell us that
	\beq
		\label{eq:3.25}
		\Big| A_{\nu}(y) - \big[ \phi(x) + \nabla\phi(x) \cdot (y-x) \big] \Big| \leq C\ep r_{\nu} \text{ for } y \in Q_{\nu}^{*},
	\eeq
	and
	\beq
		\label{eq:3.26}
		\big| \nabla A_{\nu} - \nabla\phi(x) \big| \leq C\ep.
	\eeq
	Since $x \in Q_{\nu}^{*}$. we learn from \eqref{eq:3.25} and \eqref{eq:2.15} that
	\beq
		\label{eq:3.27}
		\big| A_{\nu}(x) - \phi(x) \big| \leq C\ep\delta(x).
	\eeq
	Estimates \eqref{eq:3.26} and \eqref{eq:3.27} (and an observation that $x \in U$) are the conclusions of Lemma \ref{l:3.2}.
\end{proof}

\section{\textsc{A Partition of Unity}} \label{s:4}
\setcounter{thm}{0}
\setcounter{equation}{0}
Our plan is to patch together the map $\phi$ and the Euclidean motion $A_{\nu}$, using a partition of unity on $\bbR^{n}$. Note that the $\Theta_{\nu}$ in Section \ref{s:2} sum to 1 only on $\bbR^{n} \backslash E$.

Let $\eta > 0$ be a small enough number. Let $\chi(t)$ be a $C^{\infty}$ function on $\bbR$, having the following properties.
\beq
	\label{eq:4.1}
	\left\{\begin{array}{rl}
		0 \leq \chi(t) \leq 1 & \text{ for all } t; \\
		\chi(t) = 1 & \text{ for } t \leq \eta; \\
		\chi(t) = 0 & \text{ for } t \geq 2\eta; \\
		|\chi'(t)| \leq C\eta^{-1} & \text{ for all } t.
	\end{array}\right.
\eeq
We define
\beq
	\label{eq:4.2}
	\wt{\Theta}_{in}(x) = \chi\big(\delta(x)\big) \text{ and (for each $\nu$) } \wt{\Theta}_{\nu}(x) = \big( 1 - \wt{\Theta}_{in}(x) \big) \cdot \Theta_{\nu}(x) \text{ for } x \in \bbR^{n}.
\eeq
Thus
\beq
	\label{eq:4.3}
	\wt{\Theta}_{in}, \wt{\Theta}_{\nu} \in C^{\infty}(\bbR^{n}), \qquad \wt{\Theta}_{in} \geq 0 \text{ and } \wt{\Theta}_{\nu} \geq 0 \text{ on } \bbR^{n};
\eeq
\beq
	\label{eq:4.4}
	\wt{\Theta}_{in}(x) = 1 \text{ for } \delta(x) \leq \eta;
\eeq
\beq
	\label{eq:4.5}
	\text{supp}\,\wt{\Theta}_{in} \subset \{ x \in \bbR^{n}: \delta(x) \leq 2\eta \};
\eeq
\beq
	\label{eq:4.6}
	\text{supp}\,\wt{\Theta}_{\nu} \subset Q_{\nu}^{*} \text{ for each } \nu;
\eeq
and
\beq
	\label{eq:4.7}
	\wt{\Theta}_{in} + \sum_{\nu}\wt{\Theta}_{\nu} = 1 \text{ everywhere on } \bbR^{n}.
\eeq
Note that \eqref{eq:2.2} and \eqref{eq:4.6} yield the following.
\beq
	\label{eq:4.8}
	\text{Any given $x \in \bbR^{n}$ belongs to supp$\,\wt{\Theta}_{\nu}$ for at most $C$ distinct $\nu$}.
\eeq
In view of \eqref{eq:4.5}, we have
\beq
	\label{eq:4.9}
	\text{supp}\,\wt{\Theta}_{in} \subset U,
\eeq
if $\eta$ is small enough. This tells us in particular that $\wt{\Theta}_{in}(x)\cdot\phi(x)$ is a well-defined map from $\bbR^{n}$ to $\bbR^{n}$.

We establish the basic estimates for the gradients of $\wt{\Theta}_{in},\,\wt{\Theta}_{\nu}$. By \eqref{eq:4.4}, \eqref{eq:4.5} we have $\nabla\wt{\Theta}_{in}(x) = 0$ unless $\eta < \delta(x) < 2\eta$. For $\eta < \delta(x) < 2\eta$, we have
\[
	\big| \nabla\wt{\Theta}_{in}(x) \big| = \big| \chi'\big( \delta(x) \big) \big| \cdot \big| \nabla\delta(x) \big| \leq C\eta^{-1}
\]
by \eqref{eq:4.1} and \eqref{eq:2.9}. Therefore,
\beq
	\label{eq:4.10}
	\big| \nabla\wt{\Theta}_{in}(x) \big| \leq C(\delta(x))^{-1} \text{ for all } x \in \bbR^{n} \backslash E.
\eeq
We turn our attention to $\nabla\wt{\Theta}_{\nu}(x)$. Recall that $0 \leq \Theta_{\nu}(x) \leq 1$ and $0 \leq \wt{\Theta}_{in}(x) \leq 1$ for all $x \in \bbR^{n}$. Moreover, \eqref{eq:2.1}, \eqref{eq:2.4}, \eqref{eq:2.5} and \eqref{eq:2.7} together yield
\[
	\big| \nabla\Theta_{\nu}(x) \big| \leq C\big( \delta(x) \big)^{-1} \text{ for all } x \in \bbR^{n}\backslash E \text{ and  for all } \nu.
\]

The above remarks \big(including \eqref{eq:4.10}\big), together with the definition \eqref{eq:4.2} of $\wt{\Theta}_{\nu}$, tell us that
\beq
	\label{eq:4.11}
	\big| \nabla\wt{\Theta}_{\nu}(x) \big| \leq C(\delta(x))^{-1} \text{ for } x \in \bbR^{n}\backslash E, \text{ each } \nu.
\eeq

\section{\textsc{Extending the Map}} \label{s:5}
\setcounter{thm}{0}
\setcounter{equation}{0}
We now define
\beq
	\label{eq:5.1}
	\Phi(x) = \wt{\Theta}_{in}(x)\cdot\phi(x) + \sum_{\nu}\wt{\Theta}_{\nu}(x)\cdot A_{\nu}(x) \text{ for all } x \in \bbR^{n}.
\eeq
This makes sense, thanks to \eqref{eq:4.8} and \eqref{eq:4.9}. Moreover, $\Phi: \bbR^{n} \to \bbR^{n}$ is a $C^{1}$-map. We will prove that $\Phi$ satisfies all the conditions \ref{i:0.4} $\cdots$ \ref{i:0.9} of Theorem~\ref{t:0.1}.

First of all, for $\delta(x) < \eta$, \eqref{eq:4.3}, \eqref{eq:4.4}, \eqref{eq:4.7} give $\wt{\Theta}_{in}(x) = 1$ and all $\wt{\Theta}_{\nu}(x) = 0$; hence \eqref{eq:5.1} gives $\Phi(x) = \phi(x)$. Thus, $\Phi$ satisfies \ref{i:0.5}.

Next suppose $d(x) \geq c_{0}$. Then $\delta(x) > c > 2\eta$ if $\eta$ is small enough; hence $\wt{\Theta}_{in}(x) = 0$ and $\wt{\Theta}_{\nu}(x) = \Theta_{\nu}(x)$ for each $\nu$. (See \eqref{eq:4.3} and \eqref{eq:4.5}.) Also, \eqref{eq:2.10} shows that $\beta_{\nu} > c_{3}$ for all $\nu$ such that $x \in \text{supp}\,\Theta_{\nu}$. For such $\nu$, we have defined $A_{\nu} = A_{\infty}$; see \eqref{eq:3.3}. Hence, in this case,
\[
	\Phi(x) = \sum_{\nu}\Theta_{\nu}(x)\cdot A_{\infty}(x) = A_{\infty}(x),
\]
thanks to \eqref{eq:2.6}. Thus, $\Phi$ satisfies \ref{i:0.6}.

Next, suppose $\phi \in C^{m}(U)$ for some given $m \geq 1$. Then since $\wt{\Theta}_{in}$ and each $\wt{\Theta}_{\nu}$ belong to $C^{\infty}(\bbR^{n})$, we learn from \eqref{eq:4.8}, \eqref{eq:4.9} and \eqref{eq:5.1} that $\Phi: \bbR^{n} \to \bbR^{n}$ is a $C^{m}$ map.

Similarly, if $\phi \in C^{\infty}(U)$, then $\Phi: \bbR^{n} \to \bbR^{n}$ is a $C^{\infty}$ map. Thus, $\Phi$ satisfies \ref{i:0.8} and \ref{i:0.9}.

It remains to show that $\Phi$ satisfies \ref{i:0.4} and \ref{i:0.7}. To establish these assertions, we first control $\nabla\Phi$.

\begin{lem}
	\label{l:5.1}
	\addcontentsline{toc}{subsection}{Lemma 5.1}
	For all $x \in \bbR^{n}$ such that $\delta(x) \leq 2\eta$, we have
	\[
		\big| \nabla\Phi(x) - \nabla\phi(x) \big| \leq C\ep.
	\]
\end{lem}

\begin{proof}
	We may assume $\delta(x) \geq \eta$, since otherwise we have $|\nabla\Phi(x) - \nabla\phi(x)| = 0$ by \ref{i:0.5}. For $\delta(\ul{x}) \leq 3\eta$, we have $\ul{x} \in U$, and \eqref{eq:5.1} gives
	\beq
		\label{eq:5.2}
		\Phi(\ul{x}) - \phi(\ul{x}) = \sum_{\nu}\wt{\Theta}_{\nu}(\ul{x})\big[ A_{\nu}(\ul{x}) - \phi(\ul{x}) \big],
	\eeq
	since $\phi(\ul{x}) = \wt{\Theta}_{in}(\ul{x})\phi(\ul{x}) + \sum_{\nu}\wt{\Theta}_{\nu}(\ul{x})\phi(\ul{x})$. If $	\delta(x) \leq 2\eta$, then \eqref{eq:5.2} holds on a neighborhood of $x$; hence
	\beq
		\label{eq:5.3}
		\nabla\Phi(x) - \nabla\phi(x) = \sum_{\nu}\nabla\wt{\Theta}_{\nu}(x) \cdot \big[ A_{\nu}(x) - \phi(x) \big] + \sum_{\nu}\wt{\Theta}_{\nu}(x)\cdot \big[ \nabla A_{\nu} - \nabla\phi(x) \big].
	\eeq
	There are at most $C$ nonzero terms on the right in \eqref{eq:5.3}, thanks to \eqref{eq:4.8}. Moreover, if $\eta$ is small enough, then Lemma \ref{l:3.2} and \eqref{eq:4.6} show that $|A_{\nu}(x) - \phi(x)| \leq C\ep\delta(x)$ and $|\nabla A_{\nu} - \nabla\phi(x)| \leq C\ep$ whenever supp$\,\wt{\Theta}_{\nu} \ni x$. Also, for each $\nu$, we have $0 \leq \wt{\Theta}_{\nu}(x) \leq 1$ by \eqref{eq:4.3} and \eqref{eq:4.7}; and $|\nabla\wt{\Theta}_{\nu}(x)| \leq C\cdot(\delta(x))^{-1}$, by \eqref{eq:4.11}. Putting these estimates into \eqref{eq:5.3}, we obtain the conclusion of Lemma \ref{l:5.1}.
\end{proof}

\begin{lem}
	\label{l:5.2}
	\addcontentsline{toc}{subsection}{Lemma 5.2}
	Let $x \in Q_{\mu}^{*}$, and suppose $\delta(x) > 2\eta$. Then
	\[
		\big| \nabla\Phi(x) - \nabla A_{\mu} \big| \leq C\ep.
	\]
\end{lem}

\begin{proof}
	Since $\delta(x) > 2\eta$, we have $\wt{\Theta}_{in}(x) = 0$, $\nabla\wt{\Theta}_{in} = 0$, and $\wt{\Theta}_{\nu}(x) = \Theta_{\nu}(x)$, $\nabla\wt{\Theta}_{\nu}(x) = \nabla\Theta_{\nu}(x)$ for all $\nu$; see \eqref{eq:4.5} and \eqref{eq:4.3}. Hence, \eqref{eq:5.1} yields
	\[
		\nabla\Phi(x) = \sum_{\nu} \nabla\Theta_{\nu}(x)A_{\nu}(x) + \sum_{\nu}\Theta_{\nu}(x)\nabla A_{\nu}.
	\]
	Since also
	\[
		\nabla A_{\mu} = \sum_{\nu}\nabla\Theta_{\nu}(x)A_{\mu}(x) + \sum_{\nu}\Theta_{\nu}(x)\nabla A_{\mu},
	\]
	\big(as $\sum_{\nu}\nabla\Theta_{\nu}(x) = 0$, $\sum_{\nu}\Theta_{\nu}(x) = 1\big)$, we have
\beq
	\label{eq:5.4}
	\nabla\Phi(x) - \nabla A_{\mu} = \sum_{\nu} \nabla\Theta_{\nu}(x) \cdot \big[ A_{\nu}(x) - A_{\mu}(x) \big] + \sum_{\nu} \Theta_{\nu}(x) \cdot \big[ \nabla A_{\nu} - \nabla A_{\mu} \big].
\eeq
There are at most $C$ nonzero terms on the right in \eqref{eq:5.4}, thanks to \eqref{eq:4.8}. By \eqref{eq:2.1}, \eqref{eq:2.4}, \eqref{eq:2.5} and \eqref{eq:2.7}, we have $|\nabla\Theta_{\nu}(x)| \leq C(\delta(x))^{-1}$; and \eqref{eq:2.3}, \eqref{eq:2.6} yield $0 \leq \Theta_{\nu}(x) \leq 1$. Moreover, whenever $Q_{\nu}^{*} \ni x$, Lemma \ref{l:3.1} gives $|A_{\nu}(x) - A_{\mu}(x)| \leq C\ep\delta(x)$, and $|\nabla A_{\mu} - \nabla A_{\nu}| \leq C\ep$. When $Q_{\nu}^{*} \notniFromTxfonts x$, we have $\Theta_{\nu}(x) = 0$ and $\nabla\Theta_{\nu}(x) = 0$, by \eqref{eq:2.4}. Using the above remarks to estimate the right-hand side of \eqref{eq:5.4}, we obtain the conclusion of Lemma \ref{l:5.2}.
\end{proof}

Using Lemma \ref{l:5.1} and \ref{l:5.2}, we can show that
\beq
	\label{eq:5.5}
	(1-C\ep)I \leq \big( \nabla\Phi(x) \big)^{+}\big( \nabla\Phi(x) \big) \leq (1+C\ep)I \text{ for all } x \in \bbR^{n}.
\eeq
Indeed, if $\delta(x) \leq 2\eta$, then \eqref{eq:5.5} follows from Lemma \ref{l:1.3} and \ref{l:5.1}. If instead $\delta(x) > 2\eta$, then $x \in \bbR^{n} \backslash E$, hence $x \in Q_{\mu}$ for some $\mu$. Estimate \eqref{eq:5.5} then follows from Lemma \ref{l:5.2}, since $(\nabla A_{\mu})^{+}(\nabla A_{\mu}) = I$ for the Euclidean motion $A_{\mu}$. Thus, \eqref{eq:5.5} holds in all cases.

From \eqref{eq:5.5} and \ref{i:0.6}, together with Lemma \ref{l:1.5}, we see that
\beq
	\label{eq:5.6}
	\text{$\Phi:\bbR^{n} \to \bbR^{n}$ is one-to-one and onto, hence $\Phi^{-1}:\bbR^{n} \to \bbR^{n}$ is a $C^{1}$ diffeomorphism},
\eeq
by \eqref{eq:5.5}. Thus $\Phi$ satisfies \ref{i:0.7}. It remains only to prove \ref{i:0.4}.

To do so, we use \eqref{eq:5.5} and \eqref{eq:5.6} as follows. Let $x,y \in \bbR^{n}$. Then $|x-y|$ is the minimum of length($\Gamma$) over all $C^{1}$ curves $\Gamma$ joining $x$ to $y$. Also, by \eqref{eq:5.6}, $|\Phi(x) - \Phi(y)|$ is the infimum of length\big($\Phi(\Gamma)$\big) over all $C^{1}$ curves $\Gamma$ joining $x$ to $y$. For each $\Gamma$, \eqref{eq:5.5} yields
\[
	(1-C\ep) \,\text{length}\,(\Gamma) \leq \,\text{length}\,\big(\Phi(\Gamma)\big) \leq (1+C\ep) \,\text{length}\,(\Gamma).
\]
Taking the minimum over all $\Gamma$, we conclude that $\Phi$ satisfies \ref{i:0.4}, completing the proof of our theorem. \hfill $\blacksquare$

%
%
%
%
%

\noindent
{\bf Acknowledgment:} Support from the Department of Mathematics at Princeton, the National Science Foundation, Mathematical Reviews, The Department of Defense and the University of the Witwatersrand are gratefully acknowledged.

\end{document}